\documentclass[10pt,reqno]{amsart}
\usepackage{amsmath, amssymb, amsthm} 
\usepackage[pagewise]{lineno}\linenumbers
\nolinenumbers
\usepackage{mathrsfs}
\usepackage[breaklinks]{hyperref}
\usepackage{tikz}
\usepackage{mathrsfs}
\usepackage{multirow}

\usepackage{latexsym}
\usepackage{upref, eucal}
\usepackage{pgfplots}
\usepackage{rotating, tikz}
\usetikzlibrary{matrix,arrows,backgrounds}
\usepackage[headheight=110pt,top=1in, bottom=.9in, left=0.8in, right=0.8in]{geometry}

\usepackage{url}
\usepackage{amssymb}

\newcommand{\df}{\dfrac}
\newcommand{\tf}{\tfrac}

\renewcommand{\Re}{\operatorname{Re}}

\newcommand{\e}{\epsilon}

\renewcommand{\(}{\left\(}
\renewcommand{\)}{\right\)}
\renewcommand{\[}{\left\[}
\renewcommand{\]}{\right\]}
\renewcommand{\i}{\infty}
\numberwithin{equation}{section}
\theoremstyle{plain}
\newtheorem{theorem}{Theorem}[section]
\newtheorem{lemma}[theorem]{Lemma}
\newtheorem{corollary}[theorem]{Corollary}
\newtheorem{proposition}[theorem]{Proposition}

\newtheorem{remark}[]{Remark}

\makeatletter
\def\proof{\@ifnextchar[{\@oproof}{\@nproof}}
\def\@oproof[#1][#2]{\trivlist\item[\hskip\labelsep\textit{#2 Proof of\
		#1.}~]\ignorespaces}
\def\@nproof{\trivlist\item[\hskip\labelsep\textit{Proof.}~]\ignorespaces}

\makeatother

\makeatletter
\def\@tocline#1#2#3#4#5#6#7{\relax
	\ifnum #1>\c@tocdepth 
	\else
	\par \addpenalty\@secpenalty\addvspace{#2}%
	\begingroup \hyphenpenalty\@M
	\@ifempty{#4}{%
		\@tempdima\csname r@tocindent\number#1\endcsname\relax
	}{%
		\@tempdima#4\relax
	}%
	\parindent\z@ \leftskip#3\relax \advance\leftskip\@tempdima\relax
	\rightskip\@pnumwidth plus4em \parfillskip-\@pnumwidth
	#5\leavevmode\hskip-\@tempdima
	\ifcase #1
	\or\or \hskip 1em \or \hskip 2em \else \hskip 3em \fi%
	#6\nobreak\relax
	\dotfill\hbox to\@pnumwidth{\@tocpagenum{#7}}\par
	\nobreak
	\endgroup
	\fi}
\makeatother

\allowdisplaybreaks
\begin{document}
	\title[]{The Mordell-Tornheim zeta function: Kronecker limit type formula, Series Evaluations and Applications}
	\author{Sumukha Sathyanarayana}
	\author{N. Guru Sharan }
	\thanks{
	Abbreviated title: Kronecker limit type formulas and their applications}
	\address{Department of Mathematics, Central University of Karnataka, Kadaganchi, Kalaburagi, Karnataka-585367, India.}
	\email{sumukhas@cuk.ac.in, neerugarsumukha@gmail.com}
	\address{Department of Mathematics, Indian Institute of Technology Gandhinagar, Palaj, Gandhinagar, Gujarat-382355, India.}
	\email{gurusharan.n@iitgn.ac.in, sharanguru5@gmail.com}

	\begin{abstract}
 	In this paper, we establish Kronecker limit type formulas for the Mordell-Tornheim zeta function $\Theta(r,s,t,x)$ as a function of the second as well as the third arguments. As an application of these formulas, we obtain results of Herglotz, Ramanujan, Guinand, Zagier and Vlasenko-Zagier as corollaries. We show that the Mordell-Tornheim zeta function lies centrally between many modular relations in the literature, thus providing the means to view them under one umbrella. We also give series evaluations of $\Theta(r,s,t,x)$ in terms of Herglotz-Zagier function, Vlasenko-Zagier function and their derivatives. Using our new perspective of modular relations, we obtain a new infinite family of results called mixed functional equations.
	\end{abstract}
	\maketitle
	\tableofcontents
	\section{Introduction}
Let
\begin{align}
	\mathscr{D}:=\bigg\{(r,s,t) \in \mathbb{C}^3 \ \bigg| \Re(r+t)>1, \Re(s+t)>1 , \Re(r+s+t)>2\bigg\}. \label{Dregion}
\end{align}
 In the region $\mathscr{D},$ the Mordell-Tornheim zeta function is given by the absolutely convergent series  
	\begin{align}
	\zeta_{\textup{MT}} (r,s,t):=\sum_{n=1}^{\infty} \sum_{m=1}^{\infty} \frac{1}{n^{r} m^{s} (n+m)^{t}}.
	\end{align}
The function $\zeta_{\textup{MT}} (r,s,t)$ can be meromorphically continued to the entire $\mathbb{C}^3$ space, and its singular hyperplanes are given by $r+t=1-\ell, \ s+t=1-\ell,$ for any $\ell \in \mathbb{N}\cup\{0\}$ and by $r+s+t=2$. Matsumoto \cite{matsumoto} has given some details on the hyperplanes. The behavior of $\zeta_{\textup{MT}} (r,s,t)$ around its singularities is largely unexplored. In a recent work with Dixit \cite{dss2}, the authors considered its generalization
\begin{align}
	\Theta(r,s,t,x):=\sum_{n=1}^{\infty}  \sum_{m=1}^{\infty}\frac{1}{n^{r}m^{s}(n+mx)^{t}}, \label{Tdef}
\end{align}
for $x>0$ and $(r,s,t) \in \mathscr{D}$ as defined in \eqref{Dregion}. Clearly, $\Theta(r,s,t,1) =\zeta_{\textup{MT}} (r,s,t)$. One can easily observe that $\Theta(r,s,t,x)$ satisfies the following functional equations for $x>0$ and $(r,s,t) \in \mathscr{D}$:
\begin{align}
	&\Theta(r,s,t,x) =  \Theta(r-1,s,t+1,x)+x \hspace{0.1 cm} \Theta(r,s-1,t+1,x). \label{Tsplit}  \\
	&	\Theta(r,s,t,x)= x^{-t} \Theta(s,r,t,\tfrac{1}{x}).  \label{Tinv}
\end{align}

 The meromorphic continuation of $\Theta(r,s,t,x)$ can be done similarly as in the case of $\zeta_{\textup{MT}} (r,s,t)$ \cite[Theorem 1]{matsumoto}. Recently, it has been shown that \cite[Theorem 1]{dss2}, as $t \to 0$,
	\begin{align}\label{principal part equation}
		\Theta(1,1,t,x)=\frac{2}{t^2}+\frac{2\gamma-\log(x)}{t}+\gamma^2-\gamma\log(x)-\frac{\pi^2}{6}+O(|t|),
	\end{align}
	where $\gamma$ is Euler's constant. Results of this kind did not exist in the literature prior to this. Moreover, one can also derive the higher Laurent coefficients of 	$\Theta(1,1,t,x)$ as a finite sum of Arakawa and Kaneko constants \cite[Corollary 5]{arakawa-kaneko} 
	\begin{equation}
		c_k:=\frac{1}{k!}\int_{0}^{\infty}\left\{\log^{k}(y)-\log^{y}\left(\frac{y}{1-e^{-y}}\right)\right\}\left(\frac{1}{e^y-1}-\frac{1}{y}\right)\, dy,  \label{AK}
	\end{equation}
and integrals of the form 	
\begin{align}
	L_k^*(x):=	\sum_{j=0}^{k}a_{k+1, j}\vspace{1mm}	\int_{0}^{\infty} \frac{\log^{j}(y)}{y} \log \left( \frac{1-e^{-xy}}{1-e^{-y}} \right) \log \left( \frac{1-e^{-y}}{y}  \right) \, dt \label{Lk*}
\end{align}
for $ k \in \mathbb{N}$. For finer details, see the proofs of Theorem 3.4 and Proposition 3.7 in \cite{dss2}. Recently, Matsumoto, Onodera and Sahoo, have derived a general asymptotic result for the Mordell-Tornheim multiple-zeta function \cite{dilip}. 

Consider the real binary quadratic form $Q(x,y) =ax^2 +bx y+cy^2$ with discriminant $D= b^2-4ac$. For $a>0$, the associated zeta function is defined by
\begin{align*}
	\zeta_{Q}(s)=\sum_{(m,n) \ne (0,0)} Q(m,n)^{-s}
\end{align*}
where the sum runs over all the non-origin cordinates of the lattice $\mathbb{Z}\times\mathbb{Z}$. $\zeta_{Q}(s)$ converges absolutely for $\Re(s)>1$. For $D<0$, Kronecker proved that, around $s=1$,
\begin{align}
|D|^{\frac{s}{2}}	\zeta_{Q}(s)= \frac{2 \pi}{s-1} + 2 \pi \left(2 \gamma - \log 4 + H\left(\frac{-b+ \sqrt{D}}{2 a }\right)\right) + O(|s-1|), \label{KLF}
\end{align}
where $H(z) = - \log\left(y |\eta(z)|^4 \right),$ with $\eta(z)$ being the Dedekind eta function. The result \eqref{KLF} is famously known as the Kronecker limit formula. Siegel \cite{siegel} connected the Kronecker limit formula with the theory of $L$-functions. Over the century, several prominent mathematicians such as Hecke \cite{hecke}, Herglotz \cite{herglotz}, Ramachandra \cite{rama}, Stark \cite{stark} and Zagier \cite{zagier} gave important contributions to this field. Note that \eqref{principal part equation} is a Kronecker limit type formula for the function $\Theta(1,1,t,x)$ in the variable $t$, around $t=0$.

In Section 2, we give Kronecker limit type formulas for $\Theta(r,s,t,x)$ in $t$, i.e., the \emph{third variable}, in Theorem \ref{1stpolar}. For a fixed $r,s \in \mathbb{C}$, only possibilities of poles for $\Theta(r,s,t,x)$ in $t$ are at $t=1-r-\ell$, $t=1-s-\ell$, for $\ell \in \mathbb{N} \cup\{0\}$ and $t=2-r-s$. However, we show that all of these points need not have a pole necessarily. As a direct consequence of our Kronecker limit type formula, the singular part of the Laurent series of $\Theta(r,s,t,x)$ in $t$ is established. Interestingly, we show that the order of pole too varies by cases. $\Theta(r,s,t,x)$ can have double, single or no pole in $t$, depending on the choices of $r$ and $s$. This method cannot be used to obtain the Kronecker limit type formula for the $\Theta(r,s,t,x)$ in the \emph{second variable}. 
\subsection{Modular relations in the literature}

Let $\Re(x)>0$.  A transformation of the form $F(x)= F(1/x)$ is called called a modular relation. Equivalently, we can also write it as form $F(-1/z)=F(z)$, where $z\in\mathbb{H}$, the upper half complex plane.  

Ramanujan \cite{bcbad} had recorded the following interesting Modular relation, on page 220 of Lost Notebook:
For $x>0$, let
\begin{equation*}
	\phi(x):=\psi(x)+\frac{1}{2x}-\log(x),
\end{equation*}
where $\psi(x)=\frac{\Gamma'\hspace{-0.05 cm}(x)}{\Gamma(x)}$, is the logarithmic derivative of the Gamma function  $\Gamma(x)$. Then, 
\begin{align}
	\sqrt{x}\left\{\dfrac{\gamma-\log(2\pi x)}{2 x}+\sum_{n=1}^{\i}\phi(n x)\right\}
	&=\sqrt{\frac{1}{x}}\left\{\df{x(\gamma-\log\left(\frac{2\pi}{x}\right))}{2}+\sum_{n=1}^{\i}\phi\left(\frac{n}{x}\right)\right\} \notag \\
	&=-\df{1}{\pi^{3/2}}\int_0^{\i}\left|\Xi\left(\df{1}{2}y\right)\Gamma\left(\df{-1+iy}{4}\right)\right|^2
	\df{\cos\left(\tf{1}{2}y\log x\right)}{1+y^2}\, dy,  \label{w1.26}
\end{align}
where
\begin{align*}
	\Xi(y)&=\xi\left(\frac{1}{2}+i y\right),\\
	\xi(s)&=\frac{1}{2}s(s-1)\pi^{-\frac{s}{2}}\Gamma\left(\frac{s}{2}\right)\zeta(s),
\end{align*}
are the Riemann's functions. Dixit \cite[Theorem 4.1]{Analogues} obtained a generalization of the above result in the variable $z$, with the Hurwitz zeta function $\zeta(z,x)$ in the place of $\psi(x)$ in the first equality, the modular relation.

Guinand \cite{apg3} gave modular relations which are analogous to the first equality in \eqref{w1.26} involving derivatives of digamma function $\psi(x)$. For $z\in\mathbb{N}$ and $z>2$, he showed that,
	
	\begin{align}\label{guinand2}
		x^{\frac{z}{2}}\sum_{j=1}^{\infty}\psi^{(z-1)}(1+jx)=\bigg(\frac{1}{x}\bigg)^{\frac{z}{2}}\sum_{j=1}^{\infty}\psi^{(z-1)}\bigg(1+\frac{j}{x}\bigg).
	\end{align}
Also, analogously, for the first derivative, he proved the identity,
	\begin{align}\label{guinand1}
		x\sum_{j=1}^{\infty}\left(\psi'(1+jx)-\frac{1}{jx}\right)-\frac{1}{2}\log(x)=	\frac{1}{x}\sum_{j=1}^{\infty}\left(\psi'\bigg(1+\frac{j}{x}\bigg)-\frac{1}{\frac{j}{x}}\right)-\frac{1}{2}\log\bigg(\frac{1}{x}\bigg)
	\end{align}
Another identity which is closely related to \eqref{w1.26} is the two term functional equation of the Herglotz-Zagier function $F(x).$
The function $F(x)$ is defined as
\begin{equation}\label{herglotzdef}
	F(x):=\sum_{n=1}^{\infty}\frac{\psi(nx)-\log(nx)}{n}\hspace{5mm}\left(x\in\mathbb{C}\backslash(-\infty,0]\right).
\end{equation} 
This function was studied for the first time by Herglotz \cite{herglotz}. Later, Zagier studied this in connection with Kronecker limit formula for real quadratic fields \cite{zagier}. From the asymptotic expansion of $\psi(x)$ as $x\to \infty$, 
 we can see that that the series in \eqref{herglotzdef} converges absolutely for any $x\in\mathbb{C}\backslash(-\infty,0]$. Zagier \cite[Equations (7.4), (7.8)]{zagier} proved the following two- and three-term functional equations for $F(x)$:
\begin{align}
	&F(x)+F\left(\frac{1}{x}\right)=2F(1)+\frac{1}{2}\log^{2}(x)-\frac{\pi^2}{6x}(x-1)^2,\label{fe2}\\
	&F(x)-F(x+1)-F\left(\frac{x}{x+1}\right)=-F(1)+\textup{Li}_2\left(\frac{1}{1+x}\right),\label{fe1}
\end{align}
with \cite[Equation (7.12)]{zagier}
\begin{equation}\label{ef1}
	F(1)=-\frac{1}{2}\gamma^2-\frac{\pi^2}{12}-\gamma_1,
\end{equation}
where $\gamma_1$ is the first Stieltjes constant. Observe that the two-term functional equation \eqref{fe2} of Herglotz-Zagier function resembles the first equality of Ramanujan's modular transformation \eqref{w1.26}. For any integer $r
\geq2$, Vlasenko and Zagier \cite{vz} defined the higher Herglotz function $F_r(x)$, by the following absolutely convergent series:
\begin{equation}\label{higherherglotz}
	F_r(x):=\sum_{n=1}^{\infty}\frac{\psi(nx)}{n^r}\hspace{5mm}\left(x\in\mathbb{C}\backslash(-\infty,0]\right).
\end{equation} 
They further showed that $F_r(x)$ satisfies two- and three-term functional equations analogous \cite[Proposition 4]{vz} to \eqref{fe2} and \eqref{fe1}, given by,
\begin{align}
	&F_{r}(x) + (-x)^{r-1} F_{r}\left( \frac{1}{x} \right) = \zeta(r+1) \left( (-x)^{r} -\frac{1}{x} \right) - \sum_{\ell=1}^{r} \zeta(\ell) \zeta(r-\ell+1) (-x)^{\ell-1}, \label{vz2term} \\
	&F_{r}(x) - F_{r}(x+1) + (-x)^{r-1} F_{r} \left( \frac{x+1}{x} \right) = \zeta(r+1) \left( \frac{(-x)^{r}}{x+1}-\frac{1}{x} \right) - \sum_{\ell=1}^{r} \zeta(r-\ell+1, \ell) (-x)^{\ell-1}, \label{vz3term}
\end{align}
where the double zeta function $\zeta(s_1,s_2)$ is defined \cite[Section 1.1]{vz} by,
\begin{align}
	\zeta(s_1,s_2) = \sum_{n=1}^{\infty} \sum_{m=n}^{\infty} \frac{1}{m^{s_1} n^{s_2}}, \hspace{1 cm} \textup{(for $s_1\geq2$ and $s_2\geq1$).} \label{doublezetadef}
\end{align}
	
In this paper, we discover a new way to obtain these modular relations using the generalized Mordell-Tornheim zeta function $\Theta(r,s,t,x)$, and its properties \eqref{Tsplit} and \eqref{Tinv}. This facilitates us to develop a unified method to obtain the aforementioned two-term functional equations. The modular results discovered by Guinand, Ramanujan, Herglotz, Zagier and Vlasenko-Zagier are thus inherently connected in this way. Moreover, this approach shows that Mordell-Tornheim zeta function lies centrally between these modular relations. See Sections \ref{Section3} and \ref{SecTable} for details.

In the Section \ref{Section4}, we give the series evaluations of $\Theta(r,s,t, x)$ in terms of zeta values, $\psi(x)$ and its derivatives, in Theorem \ref{THM-EVAL}. We further use Theorem \ref{THM-EVAL} to get \emph{mixed functional equations} in \eqref{mixed} as a corollary, which have not appeared in the literature. We call them mixed functional equations since they contain terms of the form both 
\begin{align*}
	\sum_{m=1}^{\infty} \frac{\psi^{(a)}(mx+1)}{m^{b}}\hspace{1cm}\text{ and} \hspace{1cm}
	\sum_{m=1}^{\infty} \frac{\gamma + \psi(mx+1)}{m^{c}},
\end{align*}
for $a,b,$ and $c \in \mathbb{N}$, which appear in the Guinand's functional equation and the Vlasenko-Zagier's functional equation respectively. Of the infinite family of new identites we can obtain, one such new is, as explained in Remark \ref{REMARK1},
\begin{align}
	-\sum_{m=1}^{\infty} \frac{\psi'(mx+1)}{m^2} + 	\sum_{m=1}^{\infty} \frac{\psi'(\frac{m}{x}+1)}{m^2} + \frac{2}{x} \sum_{m=1}^{\infty} \frac{\gamma+\psi(mx+1)}{m^3} = \zeta^2(2). \label{newIDex}
\end{align} 

Also as a corollary of Theorem \ref{THM-EVAL}, we derive a Kronecker limit type formula, as stated in Theorem \eqref{KLF in s}, for $\Theta(r,s,t, x)$ in the \emph{second variable}.
	\section{Kronecker Limit type formula for $\Theta(r,s,t,x)$ in the Third variable} \label{Sec2}
We first prove the Proposition \ref{first}, and its corollary \ref{second}, before using them to prove the main results of the Section-the Kronecker Limit type formula in the \emph{third variable}, stated below in three Theorems \ref{1stpolar}, \ref{2ndpolar} and \ref{3rdpolar}. Theorems \ref{1stpolar} and \ref{2ndpolar} are subdivided into four cases each, whose casewise classification can be better understood by the flowcharts in Figure \ref{ExPic6} and \ref{ExPic7} respectively. Theorems \ref{1stpolar} and \ref{2ndpolar} pertains to singularities of the type $r+t=1-\ell$ for $\ell \in \mathbb{N} \cup \{0\}$ while Theorem \ref{3rdpolar} pertains to singularities of the type $r+s+t=2$.

	\begin{figure}[h!]
	\centering
	\begin{minipage}{.5\textwidth}
		\centering
		\begin{tikzpicture}[description/.style={fill=white,inner sep=2pt}]
			\useasboundingbox (2.3,3.5) rectangle (9.8,9.5);
			\scope[transform canvas={scale=0.8}]
			\node at (6,10.5) {$s \in \mathbb{C}$};
			\draw (5.4,10.2) rectangle (6.6,10.8);
			\node at (4,9) {$s \notin \mathbb{Z}$};
			\draw (3.4,8.7) rectangle (4.6,9.3);
			\node at (8,9) {$s \in \mathbb{Z}$};
			\draw (7.4,8.7) rectangle (8.6,9.3);
			\node at (6,7.5) {$s> r+\ell$};
			\draw (5.15,7.2) rectangle (6.85,7.8);
			\node at (10,7.5) {$s\leq r+\ell$};
			\draw (9.15,7.2) rectangle (10.85,7.8);
			\node at (5,6) {$s\ne \ell+1$};
			\draw (4.2,5.7) rectangle (5.8,6.3);
			\node at (7,6) {$s=\ell+1$};
			\draw (6.2,5.7) rectangle (7.8,6.3);
			\node at (9,6) {$s \in \mathbb{Z}\hspace{-0.07 cm} \setminus \hspace{-0.07 cm}\mathbb{N}$};
			\draw (8.2,5.7) rectangle (9.8,6.3);
			\node at (11,6) {$s\in \mathbb{N}$};
			\draw (10.2,5.7) rectangle (11.8,6.3);
			\draw (6,9.75) -- (6,10.2);
			\draw (6,9.75) -| (8,9.3);
			\draw (6,9.75) -| (4,9.3);
			\draw (8,8.7) -- (8,8.25);
			\draw (8,8.25) -| (10,7.8);
			\draw (8,8.25) -| (6,7.8);
			\draw (10,6.75) -- (10,7.2);
			\draw (10,6.75) -| (11,6.3);
			\draw (10,6.75) -| (9,6.3);
			\draw (6,6.75) -- (6,7.2);
			\draw (6,6.75) -| (7,6.3);
			\draw (6,6.75) -| (5,6.3);
			\node at (4,8.4) {Case I};
			\node at (7,5.4) {Case II};
			\node at (5,5.4) {Case I};
			\node at (9,5.4) {Case III};
			\node at (11,5.4) {Case IV};
			\endscope
		\end{tikzpicture}
		\caption{The cases for $r \in \mathbb{Z} \setminus \mathbb{N}$.}
		\label{ExPic6}
	\end{minipage}%
	\begin{minipage}{.5\textwidth}
		\centering
		\begin{tikzpicture}[description/.style={fill=white,inner sep=2pt}]
			\useasboundingbox (2.3,3) rectangle (10.8,9);
			\scope[transform canvas={scale=0.8}]
			\node at (6,10.5) {$s \in \mathbb{C}$};
			\draw (5.4,10.2) rectangle (6.6,10.8);
			\node at (4,9) {$s \notin \mathbb{Z}$};
			\draw (3.4,8.7) rectangle (4.6,9.3);
			\node at (8,9) {$s \in \mathbb{Z}$};
			\draw (7.4,8.7) rectangle (8.6,9.3);
			\node at (6,7.5) {$s> r+\ell$};
			\draw (5.15,7.2) rectangle (6.85,7.8);
			\node at (10,7.5) {$s\leq r+\ell$};
			\draw (9.15,7.2) rectangle (10.85,7.8);
			\node at (8.5,6) {$s \in \mathbb{Z}\hspace{-0.07 cm} \setminus \hspace{-0.07 cm}\mathbb{N}$};
			\draw (7.7,5.7) rectangle (9.3,6.3);
			\node at (11.5,6) {$s\in \mathbb{N}$};
			\draw (10.7,5.7) rectangle (12.3,6.3);
			\node at (10.5,4.5) {$s\ne \ell+1$};
			\draw (9.7,4.2) rectangle (11.3,4.8);
			\node at (12.5,4.5) {$s=\ell+1$};
			\draw (11.7,4.2) rectangle (13.3,4.8);
			\draw (6,9.75) -- (6,10.2);
			\draw (6,9.75) -| (8,9.3);
			\draw (6,9.75) -| (4,9.3);
			\draw (8,8.7) -- (8,8.25);
			\draw (8,8.25) -| (10,7.8);
			\draw (8,8.25) -| (6,7.8);
			\draw (10,6.75) -- (10,7.2);
			\draw (10,6.75) -| (11.5,6.3);
			\draw (10,6.75) -| (8.5,6.3);
			\draw (11.5,5.7) -- (11.5,5.25);
			\draw (11.5,5.25) -| (12.5,4.8);
			\draw (11.5,5.25) -| (10.5,4.8);
			\node at (4,8.4) {Case I};
			\node at (6,6.9) {Case I};
			\node at (8.5,5.4) {Case II};
			\node at (10.5,3.9) {Case III};
			\node at (12.5,3.9) {Case IV};
			\endscope
		\end{tikzpicture}
		\caption{The cases for $r \in \mathbb{N}$.}
		\label{ExPic7}
	\end{minipage}
\end{figure}

	\begin{proposition}\label{first}
		Let $r,s,t \in \mathbb{C}$ such that $r \not\in \mathbb{N}$, and \\ $\Re(t)> \textup{max} \, \{0,1- \Re(r), 1- \Re(s), 2-\Re(r)-\Re(s)\}$. For any $M \in \mathbb{N} \cup \{0\} $ and $x>0$, we have
		\begin{align}\label{PE}
			&\frac{1}{\Gamma(t)} \int_{0}^{\infty} y^{t-1} \textup{Li}_s (e^{-xy}) \Bigg( \textup{ Li}_r (e^{-y}) - \Gamma(1-r) y^{r-1} - \sum_{\substack{k=0 }}^{M} (-1)^k\zeta(r-k) \frac{y^k}{k!}\Bigg) \, dy \notag \\
			&=  \Theta(r,s,t,x) - \frac{\Gamma(1-r)\Gamma(t+r-1)}{ x^{t+r-1}\Gamma(t)}\zeta(t+s+r-1)    - \sum_{\substack{k=0 }}^{M} \frac{ (-1)^k}{k!} \frac{\zeta(r-k)\Gamma(t+k)}{ x^{t+k}\Gamma(t)} \zeta(t+k+s).
		\end{align} 
	\end{proposition}
	
	\begin{proof}
		For $\Re(t),x,n>0$, the definition of the Gamma function $\Gamma(t)$ upon a change of variable gives,
		\begin{align*}
			\int_{0}^{\infty} y^{t-1} e^{-(x+n)y} \, dy = (x+n)^{-t} \Gamma(t).
		\end{align*} 
		Since $\textup{Re}(r+t)>1$, divide both sides of the equation by $n^r$ and sum over $n \in \mathbb{N}$ and interchange the order of summation and integration to get,
		\begin{align}
			\int_{0}^{\infty} y^{t-1} e^{-xy} \textup{ Li}_r (e^{-y})  \, dy &= \Gamma(t) \sum_{n=1}^{\infty} \frac{1}{n^r (x+n)^{t}}, \label{SI}
		\end{align} 
		where $\textup{Li}_r(z):=\sum_{u=1}^{\infty} \frac{z^u}{u^r}$ is the polylogarithm function.  For $r \not\in \mathbb{N}$, around $y=0$, as stated in \cite[Equation 8, p.29]{Erdelyi}, we have,
		\begin{align}\label{LiE}
			\textup{ Li}_r (e^{-y}) = \Gamma(1-r) y^{r-1}+ \sum_{\substack{k=0 }}^\infty (-1)^k\zeta(r-k) \frac{y^k}{k!} = O(y^{\textup{max}(0,\textup{Re}(r)-1)}). 
		\end{align}
		The bound in \eqref{LiE} along with $\Re(t)> \textup{max} \, \{0, 1-\Re(r)\}$ ensures the absolute convergence, which in turn allows the interchange using  \cite[Theorem 2.1]{temme}. 
		Using \eqref{SI} and the definition of $\Gamma(t)$ twice, we see that,
		\begin{align}\label{after eva notnat}
			&\int_{0}^{\infty} y^{t-1} e^{-xy} \Bigg( \textup{ Li}_r (e^{-y})- \Gamma(1-r) y^{r-1} - \sum_{\substack{k=0 }}^{M} (-1)^k\zeta(r-k) \frac{y^k}{k!}\Bigg) \, dy \notag \\
			&= \Gamma(t) \sum_{n=1}^{\infty} \frac{1}{n^r (x+n)^{t}} - \frac{\Gamma(1-r)\Gamma(t+r-1)}{x^{t+r-1}}  - \sum_{\substack{k=0 }}^{M} \frac{ (-1)^k}{k!} \frac{\zeta(r-k)\Gamma(t+k)}{x^{t+k}}.
		\end{align}
		Replace $x$ by $mx$ in \eqref{after eva notnat}, divide both the sides of resulting equation by $m^s \Gamma(t)$, sum over $m \in \mathbb{N}$ and then interchange the order of summation and integration to get,
		\begin{align}\label{final not nat}
			&\frac{1}{\Gamma(t)} \int_{0}^{\infty} y^{t-1} \textup{Li}_s (e^{-xy}) \Bigg( \textup{ Li}_r (e^{-y}) - \Gamma(1-r) y^{r-1} - \sum_{\substack{k=0 }}^{M} (-1)^k\zeta(r-k) \frac{y^k}{k!}\Bigg) \, dy \notag \\
			&= \Theta(r,s,t, x)  - \frac{\Gamma(1-r)\Gamma(t+r-1)}{ x^{t+r-1}\Gamma(t)}\zeta(t+s+r-1)    - \sum_{\substack{k=0 }}^{M} \frac{ (-1)^k}{k!} \frac{\zeta(r-k)\Gamma(t+k)}{ x^{t+k}\Gamma(t)} \zeta(t+k+s).
		\end{align} 
		We justify the interchange of the order of summation and integration as follows: From \eqref{LiE}, we have,
		\begin{align}\label{BOr}
			\textup{ Li}_r (e^{-y}) - \Gamma(1-r) y^{r-1} - \sum_{\substack{k=0 }}^{M} (-1)^k\zeta(r-k) \frac{y^k}{k!} =  O(y^{M+1}), \hspace{0.5 cm} \textup{as}\quad y \to 0^+.
		\end{align}
		Similarly,
		\begin{equation}\label{BOs}
			\textup{ Li}_s (e^{-y})=
			\begin{cases}
				O(1),\hspace{5mm}&\text{if}\hspace{1mm}\Re(s) \ge 1,\\
				O(y^{\Re(s)-1}),&\text{if}\hspace{1mm}\Re(s) < 1,
			\end{cases}
			\hspace{0.5 cm}\textup{as} \quad y \to 0^+.
		\end{equation}
		Since $\Re(t)>\text{max}\, \{ 0, 1-\Re(s)\}$, from \eqref{BOr} and \eqref{BOs} we get,
		\begin{align}\label{BO0}
			y^{t-1} \textup{Li}_s (e^{-xy}) \bigg( \hspace{-0.2 cm} \textup{ Li}_r (e^{-y}) - \Gamma(1-r) y^{r-1} - \sum_{\substack{k=0 }}^{M} (-1)^k\zeta(r-k) \frac{y^k}{k!}\bigg) = O\Big(y^{M+\epsilon }\Big), 
		\end{align}
		for any $\epsilon >0$, as $y \to 0^+$. Also,  $\textup{ Li}_s (e^{-y})= O(y^{-K})$, for any $K>1$ as $ y \to \infty$. Therefore,  
		\begin{align}\label{BO1}
			y^{t-1} \textup{Li}_s (e^{-xy}) \bigg( \hspace{-0.15 cm} \textup{ Li}_r (e^{-y}) - \Gamma(1-r) y^{r-1} - \sum_{\substack{k=0 }}^{M} (-1)^k\zeta(r-k) \frac{y^k}{k!}\bigg) = O\Big(y^{\textup{max}\{\Re(t)+M-1-K,\, \Re(t)+r-2-K\} }\Big),
		\end{align}
		for any $K>1$, as $ y \to \infty$. Using \cite[Theorem 2.1]{temme} with the facts in \eqref{BO0} and \eqref{BO1}, we justify the interchange in \eqref{final not nat} to complete the proof.
	\end{proof}
	We now give the following result, analogous to Proposition \ref{first} for $r \in\mathbb{N}$ as a Corollary. 
	\begin{corollary} \label{second}
		Let $r \in \mathbb{N}$ and $s,t \in \mathbb{C}$ such that $\Re(t)> \textup{max} \, \{0,1- \Re(r), 1- \Re(s), 2-\Re(r)-\Re(s)\}$. For any $M \in \mathbb{N} \ \cup \{0\} $ with $M \ge r-1$, and $x>0$ we have,
		\begin{align}
			&\frac{1}{\Gamma(t)} \int_{0}^{\infty} y^{t-1} \textup{Li}_s (e^{-xy}) \Bigg( \textup{ Li}_r (e^{-y}) - (-1)^{r-1} (H_{r-1} -\log(y)) \frac{y^{r-1}}{(r-1)!} - \sum_{\substack{k=0 \\ k\ne r-1}}^{M} (-1)^k\zeta(r-k) \frac{y^k}{k!} \Bigg) \, dy \notag \\
			&=  \Theta(r,s,t,x)  - \sum_{\substack{k=0 \\ k\ne r-1}}^{M} \frac{ (-1)^k}{k!} \frac{\zeta(r-k)\Gamma(t+k)}{ x^{t+k}\Gamma(t)} \zeta(t+k+s) \notag  \\
			& \quad +  \frac{(-1)^{r}\Gamma(t+r-1)}{x^{t+r-1} (r-1)!\Gamma(t)}\Big(\zeta(t+s+r-1) \left( H_{r-1}  - \psi(t+r-1) + \log(x) \right)- \zeta'(t+s+r-1)\Big). \label{CE}
		\end{align} 
	\end{corollary}
	
	\begin{proof}
		Let $a\in \mathbb{N}$. Using the Laurent series expansions of $\Gamma(1-r)$ and $\zeta(r-a+1)$ around $r=a$ and simplifying, we get,
		\begin{align}
			&\lim_{r \to a} \Bigg( \textup{ Li}_r (e^{-y}) - \Gamma(1-r) y^{r-1} - \sum_{\substack{k=0 }}^{M} (-1)^k\zeta(r-k) \frac{y^k}{k!}\Bigg) \notag \\ 
			&= \textup{ Li}_a (e^{-y}) + (-1)^{a} (H_{a-1} -\log(y)) \frac{y^{a-1}}{(a-1)!} - \sum_{\substack{k=0 \\ k\ne a-1}}^{M} (-1)^k\zeta(a-k) \frac{y^k}{k!} \label{rLim}
		\end{align}
		Take limit as $r \to a$ on both sides of \eqref{PE}. Invoke \cite[Theorem 2.2]{temme} with the bounds given in \eqref{BO0} and \eqref{BO1} to see the integral on the left-hand side is analytic in $r$, in the strip $a-1< \Re(r)<a+1$ with a removable singularity at $a$. Hence, interchanging the order of limit and integration, we get,
		\begin{align*}
			&\frac{1}{\Gamma(t)} \int_{0}^{\infty} y^{t-1} \textup{Li}_s (e^{-xy}) \bigg( \textup{Li}_a (e^{-y}) + (-1)^{a} (H_{a-1} -\log(y)) \frac{y^{a-1}}{(a-1)!} - \sum_{\substack{k=0 \\ k\ne a-1}}^{M} (-1)^k\zeta(a-k) \frac{y^k}{k!} \bigg) \, dy \notag \\
			&= \Theta(a,s,t,x) - \sum_{\substack{k=0  \\ k \ne a-1}}^{M} \frac{ (-1)^k}{k!} \frac{\zeta(a-k)\Gamma(t+k)}{ x^{t+k}\Gamma(t)} \zeta(s+t+k)   \notag\\
			& - \lim_{r \to a} \bigg( \frac{\Gamma(1-r)\Gamma(t+r-1)}{ x^{t+r-1}\Gamma(t)}\zeta(s+t+r-1)- \frac{(-1)^{a}\zeta(r-a+1)\Gamma(t+a-1)}{ x^{t+a-1}(a-1)!\Gamma(t)} \zeta(s+t+a-1) \bigg)  .
		\end{align*}
		To evaluate the limit in the above equation, use the Laurent series expansions of the respective functions around $r=a$. Finally, replace $a$ by $r$ in the resultant to get \eqref{CE}.
	\end{proof}
	 
	We are now equipped to prove the Kronecker Limit type formula in \emph{third variable}, around $t=1-r-\ell$. 
	
	\begin{theorem}\label{1stpolar}
		Let $r \in \mathbb{Z}\setminus \mathbb{N}$ and $\ell \in \mathbb{N} \cup \{0\}$ be arbitrary such that $\ell\geq 1-r$. Then, the following hold.\\
		\textbf{\textup{Case I:}} For a fixed $s \in \mathbb{Z}$ such that $s>r+\ell$ and $s\ne \ell+1$, or $s \in \mathbb{C} \setminus \mathbb{Z}$, we have, around $t=1-r-\ell$,
		\begin{align}
			\Theta(r,s,t,x) &= (-1)^{r-1}\frac{ (r+\ell-1)!}{ \ell!}x^\ell \Gamma(1-r) \zeta(s-\ell) \notag \\
			& \quad + \sum_{\substack{k=0}}^{r+\ell-1} \frac{(r+\ell-1)!\zeta(r-k)}{ x^{1-r-\ell+k}k!(r+\ell-k-1)!} \zeta(-r+s-\ell+k+1) + O(|t-(1-r-\ell)|) \label{T1C1}
		\end{align}
		\textbf{\textup{Case II:}} For a fixed $s \in \mathbb{Z}$ such that $s>r+\ell$ and $s=\ell+1$, we have,  around $t=1-r-\ell$,
		\begin{align}
			\Theta(r,s,t,x)
			&= (-1)^{r-1}(-r)! x^\ell  \frac{ (r+\ell-1)!}{\ell! (t-(1-r-\ell))} +(-1)^{r-1} (-r)! x^\ell \frac{ (r+\ell-1)!}{\ell! }\left(\gamma-\log(x)+H_{\ell}-H_{r+\ell-1}\right) \notag \\
			& \quad + \sum_{\substack{k=0}}^{r+\ell-1} \binom{r+\ell-1}{k} \frac{\zeta(r-k)}{ x^{k-r-\ell+1}} \zeta(k-r+2) + O(|t-(1-r-\ell)|). \label{T1C2}
		\end{align}
		\textbf{\textup{Case III:}} For a fixed $s \in \mathbb{Z} \setminus \mathbb{N}$ such that $s \leq r+\ell$, we have,  around $t=1-r-\ell$,
		\begin{align}
			\Theta(r,s,t,x)
			&= (-1)^{r-1}x^{\ell}\frac{(-r)!(r+\ell-1)!}{ \ell!}\zeta(s-\ell)   + \sum_{\substack{k=0}}^{r+\ell-1} \frac{\zeta(r-k)(r+\ell-1)!}{ x^{k-r-\ell+1}k!(r+\ell-k-1)!} \zeta(s+1-r-\ell+k) \notag\\
			&\quad +(-1)^{s-1}x^{s-1}\frac{(r+\ell-1)!(-s)!}{ (r+\ell-s)!} \zeta(s-\ell)+ O(|t-(1-r-\ell)|). \label{T1C3}
		\end{align}
		\textbf{\textup{Case IV:}} For a fixed $s \in \mathbb{N}$ such that $s \leq r+\ell$, we have,  around $t=1-r-\ell$,
		\begin{align}
			&\Theta(r,s,t,x) =   x^{s-1} \frac{(r+\ell-1)!}{ (r+\ell-s)! (s-1)!}   \frac{\zeta(s-\ell)}{(t-(1-r-\ell))} \notag \\
			&\quad + x^{s-1} \frac{(r+\ell-1)!}{ (r+\ell-s)! (s-1)!} (\gamma -H_{r+\ell-1}+H_{s-1}-\log(x)) \zeta(s-\ell) -(-1)^{r}x^{\ell}\frac{(-r)!(r+\ell-1)!}{ \ell!}\zeta(s-\ell)  \notag \\
			&\quad + \sum_{\substack{k=0 \\ k \ne r+\ell-s}}^{r+\ell-1} \frac{(r+\ell-1)!\zeta(r-k)}{(r+\ell-k-1)! x^{k-r-\ell+1}k!} \zeta(s-r-\ell+k+1) + O(|t-(1-r-\ell)|) \label{T1C4}
		\end{align}
	\end{theorem}
	\begin{proof}
		Let $r \in \mathbb{Z}\setminus \mathbb{N}$ and $\ell \in \mathbb{N} \cup \{0\}$ such that $\ell\geq 1-r$ be fixed and $s \in \mathbb{C}$.
		Then, for $t>-r-\ell$, \eqref{BOr} and \eqref{BOs} gives, 
		\begin{align}
			y^{t-1} \textup{Li}_s (e^{-xy}) \bigg( \hspace{-0.2 cm} \textup{ Li}_r (e^{-y}) - \Gamma(1-r) y^{r-1} - \sum_{\substack{k=0 }}^{r+\ell+ \lfloor|\textup{Re}(s)|\rfloor +1} (-1)^k\zeta(r-k) \frac{y^k}{k!}\bigg) = O\left(y^{\epsilon-1}\right), \label{BOt}
		\end{align} 
		for some $\e >0$. For $M=r+\ell+ \lfloor|\textup{Re}(s)|\rfloor +1$, from \eqref{BO1} and \eqref{BOt}, the integral on the left-hand side of \eqref{PE} is now defined and analytic in the extended region $\textup{Re}(t)>-r-\ell$  by using \cite[Theorem 2.2]{temme}. Hence, due to the pole of $\Gamma(t)$ at $t=1-r-\ell$, we see
		\begin{align*}
			\lim_{t \to 1-r-\ell} \Bigg( \frac{1}{\Gamma(t)} \int_{0}^{\infty} y^{t-1} \textup{Li}_s (e^{-xy}) \bigg( \hspace{-0.2 cm} \textup{ Li}_r (e^{-y}) - \Gamma(1-r) y^{r-1} - \sum_{\substack{k=0 }}^{r+\ell+ \lfloor|\textup{Re}(s)|\rfloor +1} (-1)^k\zeta(r-k) \frac{y^k}{k!}\bigg) \, dy \Bigg) =0.
		\end{align*} 
		Hence, from \eqref{PE}
		\begin{align}
			\Theta(r,s,t,x) &= \frac{\Gamma(1-r)\Gamma(t+r-1)}{ x^{t+r-1}\Gamma(t)}\zeta(r+s+t-1) \notag \\
			& \quad + \sum_{\substack{k=0}}^{r+\ell+ \lfloor|\textup{Re}(s)|\rfloor +1} \frac{(-1)^k\zeta(r-k)\Gamma(t+k)}{ x^{t+k}k!\Gamma(t)} \zeta(s+t+k) + O(|t-(1-r-\ell)|). \label{M1R}
		\end{align}
		To get the required result for $\Theta(r,s,t,x)$, we analyze the behaviour of the right-hand side of \eqref{M1R} around $t=1-r-\ell$. Since $\Gamma(t)$ has a pole at $t=1-r-\ell$, only the terms with a pole in the numerator survive as $t \to 1-r-\ell$. We handle the remaining calculations in cases, as mentioned in statement of this Theorem. \\ 
		\textbf{Case I:} 
			Note that, $\Gamma(t+k)$ has a simple pole only for $k \leq r+\ell-1$. For any $k$, $\zeta(s+t+k)$ has no pole since $s>r+\ell$ or $s \notin \mathbb{Z}$.
		$\Gamma(t+r-1)$ has a simple pole.
		$\zeta(r+s+t-1)$ has no pole since $s \ne \ell+1$.
		%
		Use the functional equation $\Gamma(t+1)=t \Gamma(t)$ repeatedly to get,
		\begin{align*}
			\frac{\Gamma(t+k)}{\Gamma(t)}=t (t+1)\cdots(t+k-1). 
		\end{align*}
		Taking limit as $t \to 1-r-\ell$ on both sides of the above equation, for $0\leq k \leq r+\ell-1$ we get,
		\begin{align}
			\lim_{t\to 1-r-\ell} \left( \frac{\Gamma(t+k)}{\Gamma(t)} \right) = (1-r-\ell)(2-r-\ell)\cdots (k-r-\ell) = (-1)^{k} \frac{(r+\ell-1)!}{(r+\ell-k-1)!}. \label{Glim1}
		\end{align}
		Also, since $r \in \mathbb{Z} \setminus \mathbb{N}$, 
		\begin{align*}
			\frac{\Gamma(t+r-1)}{\Gamma(t)} = \frac{\Gamma(t+r-1)}{(t-1)(t-2) \cdots (t+r-1) \Gamma(t+r-1)}  = \frac{1}{(t-1)(t-2) \cdots (t+r-1)}
		\end{align*}
		Hence, we can see,
		\begin{align}
			\lim_{t \to 1-r-\ell} \left( \frac{\Gamma(t+r-1)}{\Gamma(t)}  \right) = (-1)^{r-1} \frac{(r+\ell-1)!}{\ell!}. \label{Glim2}
		\end{align}
		Use \eqref{Glim1} and \eqref{Glim2} in \eqref{M1R} to get the required result \eqref{T1C1}.\\
		\textbf{Case II:} Note that, $\Gamma(t+k)$ has a simple pole only for $k \leq r+\ell-1$. For any $k$, $\zeta(s+t+k)$ has no pole since $s>r+\ell$. $\Gamma(t+r-1)$ has a simple pole.
		$\zeta(r+s+t-1)$ has a simple pole since $s = \ell+1$. Also, observe that, for any $k, r \in \mathbb{N}$ such that $k \leq r-1$, around $t=1-r$, from Laurent series expansions of the $\Gamma(z)$, we have 
		\begin{align}
			\frac{\Gamma(t+r-1)}{\Gamma(t)}= (-1)^{r-1}\frac{ (r+\ell-1)!}{\ell!}-(-1)^{r}\frac{ (r+\ell-1)!}{\ell!} (H_{\ell}-H_{r+\ell-1}) (t-(1-r-\ell)) + O(|t-(1-r-\ell)|^2). \label{Glim3}
		\end{align}
		Also, the Laurent series of $x^{r+t-1}$ with additional term around $r+t=1-\ell$ is as follows,
		\begin{align}
			\frac{1}{x^{t+r-1}} =x^{\ell} - x^{\ell} \log(x) (t-(1-r-\ell)) + O(|t-(1-r-\ell)|^2). \label{xlim}
		\end{align}
		The well-known Laurent series expansion of $\zeta(z)$,
		\begin{align}
			\zeta(t+r+\ell)= \frac{1}{(t-(1-r-\ell))} + \gamma - \gamma_1 (t-(1-r-\ell)) +O(|t-(1-r-\ell)|^2). \label{zlim}
		\end{align}
		Combining the equations \eqref{Glim1}, \eqref{Glim3}, \eqref{xlim} and \eqref{zlim} in \eqref{M1R}, we get the required result \eqref{T1C2}.\\ 
		\textbf{Case III:} Note that, $\Gamma(t+k)$ has a simple pole only for $k \leq r+\ell-1$. Observe that $\zeta(s+t+k)$ has a simple pole only for $k=r+\ell-s$, since $s\leq r+\ell$. Also note that $r+\ell-s> r+\ell-1$ since $s \in \mathbb{Z} \setminus  \mathbb{N}$. $\Gamma(t+r-1)$ has a simple pole. $\zeta(r+s+t-1)$ has no pole since $s \leq r+ \ell$. Hence, from \eqref{M1R}, we get the required result \eqref{T1C3}.
		\\ 
		\textbf{Case IV:} Note that, $\Gamma(t+k)$ has a simple pole only for $k \leq r+\ell-1$. Observe that $\zeta(s+t+k)$ has a simple pole only for $k=r+\ell-s$, since $s\leq r+\ell$. Also note that $r+\ell-s\leq r+\ell-1$ since $s \in \mathbb{N}$. $\Gamma(t+r-1)$ has a simple pole. $\zeta(r+s+t-1)$ has no pole since $s \leq r+ \ell$.
	Hence, from \eqref{M1R}, we get the required result \eqref{T1C4}.
	\end{proof}

	\begin{theorem}\label{2ndpolar}
		Let $r \in \mathbb{N}$ and $\ell \in \mathbb{N} \cup \{0\}$ be arbitrary such that $\ell \ge 1-r$. Then, the following hold.\\
		\textbf{\textup{Case I:}} For a fixed $s \in \mathbb{Z}$ such that $s>r+\ell$ , or $s \in \mathbb{C} \setminus \mathbb{Z}$, we have,
			\begin{align}
			\Theta(r,s,t,x) 
			&= \binom{r+\ell-1}{\ell} x^{\ell} \frac{\zeta(s-\ell)}{(t-(1-r-\ell))} + \binom{r+\ell-1}{\ell} x^{\ell} \zeta(s-\ell) \left(\gamma+H_{r-1}-H_{r+\ell-1}\right) \notag  \\
			&\quad +\sum_{\substack{k=0 \\ k\ne r-1}}^{r+\ell-1} \binom{r+\ell-1}{k} x^{r+\ell-1-k} \zeta(r-k) \zeta(1+k-\ell-r+s) + O(|t-(1-r-\ell)|). \label{T2C1}
		\end{align}
		\\
		\textbf{\textup{Case II:}}  For a fixed $s \in \mathbb{Z} \setminus \mathbb{N}$ such that $s \leq r+\ell$, we have,
		\begin{align}
			\Theta(r,s,t,x) &=	 \binom{r+\ell-1}{\ell} x^{\ell} \frac{\zeta(s-\ell)}{(t-(1-r-\ell))} + \binom{r+\ell-1}{\ell} x^{\ell} \zeta(s-\ell) \left(\gamma+H_{r-1}-H_{r+\ell-1}\right) \notag  \\
			& \quad + \sum_{\substack{k=0 \\ k\ne r-1}}^{r+\ell-1} \binom{r+\ell-1}{k} x^{r+\ell-1-k} \zeta(r-k) \zeta(1+k-\ell-r+s) \notag  \\
			& \quad  + (-1)^{s-1} \frac{(r+\ell-1)!(-s)!}{(r+\ell-s)!} x^{s-1} \zeta(s-\ell) + O(|t-(1-r-\ell)|). \label{T2C2}
		\end{align}
		\\
		\textbf{\textup{Case III:}}  For a fixed $s \in \mathbb{N}$ such that $s\leq r+\ell$ and $s \ne \ell+1$, we have,
			\begin{align}
			\Theta(r,s,t,x) &= \binom{r+\ell-1}{\ell} x^{\ell} \frac{\zeta(s-\ell)}{(t-(1-r-\ell))} + \binom{r+\ell-1}{\ell} x^{\ell} \zeta(s-\ell) \left(\gamma+H_{r-1}-H_{r+\ell-1}\right) \notag  \\
			&\quad + \binom{r+\ell-1}{s-1}x^{s-1} \frac{\zeta(s-\ell)}{(t-(1-r-\ell))} + \binom{r+\ell-1}{s-1}x^{s-1} \zeta(s-\ell) \left( \gamma + H_{s-1} -H_{r+\ell-1} - \log(x) \right) \notag \\
			&\quad +\sum_{\substack{k=0 \\ k\ne r-1 \\ k\ne r+\ell-s}}^{r+\ell-1} \binom{r+\ell-1}{k} x^{r+\ell-1-k} \zeta(r-k) \zeta(1+k-\ell-r+s)  + O(|t-(1-r-\ell)|). \label{T2C3}
			\end{align}
		\\
		\textbf{\textup{Case IV:}}  For a fixed $s \in \mathbb{N}$ such that ($s\leq r+\ell$ and) $s = \ell+1$, we have,
		\begin{align}
			\Theta(r,s,t,x) & =  \binom{r+\ell-1}{\ell}  \frac{ 2 x^{\ell}}{(t-(1-r-\ell))^2} + \binom{r+\ell-1}{\ell}  \frac{  x^{\ell} \left(2 \gamma+ H_{\ell}+H_{r-1}-2 H_{r+\ell-1}-\log(x)\right)}{(t-(1-r-\ell))} \notag \\
			&\quad + \frac{1}{3}\binom{r+\ell-1}{\ell} x^{\ell} \bigg( \hspace{-0.2 cm} -\pi^2+ 3 (\gamma+ H_{\ell}-H_{r+\ell-1}-\log(x)) (\gamma+H_{r-1}-H_{r+\ell-1})+3 \psi'(r+\ell)\bigg) \notag \\
			&\quad 	+ \sum_{\substack{k=0 \\ k\ne r-1}}^{r+\ell-1} \binom{r+\ell-1}{k} x^{r+\ell-1-k} \zeta(r-k) \zeta(k-r+2) + O(|t-(1-r-\ell)|). \label{T2C4}
		\end{align}		
	\end{theorem}
	\begin{proof}
		Let $s \in \mathbb{C}$. Then, for $t>-r-\ell$, \eqref{rLim} and \eqref{BOt} gives,  
		\begin{align}
			\textup{ Li}_r (e^{-y}) - (-1)^{r-1} (H_{r-1} -\log(y)) \frac{y^{r-1}}{(r-1)!} - \sum_{\substack{k=0 \\ k\ne r-1}}^{ r+\ell+ \lfloor|\textup{Re}(s)|\rfloor +1} (-1)^k\zeta(r-k) \frac{y^k}{k!} = O(y^{\epsilon -1}) \label{BOt1}
		\end{align} 
		for some $\e >0$. As in the previous theorem, for $M=r+\ell+ \lfloor|\textup{Re}(s)|\rfloor +1$, from \eqref{BO1} and \eqref{BOt1}, the integral on the left-hand side of \eqref{CE} is now defined and analytic in the extended region $\textup{Re}(t)>-r-\ell$  by using \cite[Theorem 2.2]{temme}. Hence, due to the pole of $\Gamma(t)$ at $t=1-r-\ell$, we see
		\begin{align*}
			\lim_{t \to 1-r-\ell} \Bigg( \frac{1}{\Gamma(t)} \int_{0}^{\infty} y^{t-1} \textup{Li}_s (e^{-xy}) \bigg( \textup{ Li}_r(e^{-y}) -& (-1)^{r-1} (H_{r-1} -\log(y)) \frac{y^{r-1}}{(r-1)!} \\ &- \sum_{\substack{k=0 \\ k\ne r-1}}^{r+\ell+ \lfloor|\textup{Re}(s)|\rfloor +1} (-1)^k\zeta(r-k) \frac{y^k}{k!} \bigg) \, dy \Bigg) =0.
		\end{align*}
		Hence, from \eqref{CE},
		\begin{align}
			\Theta(r,s,t,x) & =  \frac{(-1)^{r-1}\Gamma(t+r-1)}{x^{t+r-1} (r-1)!\Gamma(t)}\Big(\zeta(t+s+r-1) \left( H_{r-1}  - \psi(t+r-1) + \log(x) \right)- \zeta'(t+s+r-1)\Big) \notag \\ 
			&\quad +\sum_{\substack{k=0 \\ k\ne r-1}}^{r+\ell+ \lfloor|\textup{Re}(s)|\rfloor +1} \frac{ (-1)^k}{k!} \frac{\zeta(r-k)\Gamma(t+k)}{ x^{t+k}\Gamma(t)} \zeta(t+k+s) + O(t-(1-r-\ell)). \label{M1N}
		\end{align}
		To get the required Kronecker limit formula for $\Theta(r,s,t,x)$, we analyze the behaviour of the right-hand side of \eqref{M1R} around $t=1-r-\ell$. Since $\Gamma(t)$ has a pole at $t=1-r-\ell$, only the terms with a pole in the numerator survive, as $t \to 1-r-\ell$. We handle the remaining calculations in cases, as mentioned in statement of this Theorem. \\ 
		\textbf{Case I:} Note that, $\Gamma(t+k)$ has a simple pole only for $k \leq r+\ell-1$. Observe that $\zeta(s+t+k)$ and $\zeta(r+s+t-1)$ have no poles since $s > r+\ell$ or $s \notin \mathbb{Z}$.  $\Gamma(t+r-1)$ has a simple pole.  $\psi(t+r-1)$ has a simple pole. From \eqref{M1N}, using \eqref{Glim1} and \eqref{Glim3}, we get the required result \eqref{T2C1}. \\ 
		\textbf{Case II:} Note that, $\Gamma(t+k)$ has a simple pole only for $k \leq r+\ell-1$. Observe that $\zeta(s+t+k)$ has a simple pole for $k =r+\ell-s$, moreover $r+\ell-s > r+\ell-1$. $\zeta(r+s+t-1)$ has no pole since $s \ne \ell+1$.  $\Gamma(t+r-1)$ has a simple pole. $\psi(t+r-1)$ has a simple pole. Also observe that the term corresponding to $k=r+\ell-s$ appears due to the pole of $\zeta(t+k+s)$ and not $\Gamma(t+k)$. From \eqref{M1N}, we get the required result \eqref{T2C2}. \\  
		\textbf{Case III:} Note that, $\Gamma(t+k)$ has a simple pole only for $k \leq r+\ell-1$. Observe that $\zeta(s+t+k)$ has a simple pole for $k =r+\ell-s$, moreover $r+\ell-s \le r+\ell-1$. $\zeta(r+s+t-1)$ has no pole since $s \ne \ell+1$.  $\Gamma(t+r-1)$ has a simple pole. $\psi(t+r-1)$ has a simple pole. As in the previous case, the term corresponding to $k=r+\ell-s$ appears due to the pole of $\zeta(t+k+s)$ and not $\Gamma(t+k)$. From \eqref{M1N}, we get the required result \eqref{T2C3}. \\  
		\textbf{Case IV:} Note that, $\Gamma(t+k)$ has a simple pole only for $k \leq r+\ell-1$. Observe that $\zeta(s+t+k)$ has no pole since $s=\ell+1$ and $ k \ne r-1$. $\zeta(r+s+t-1)$ and $\zeta'(r+s+t-1)$ have a  simple and a double pole respectively since $s = \ell+1$.  $\Gamma(t+r-1)$ has a simple pole. $\psi(t+r-1)$ has a simple pole. Also, use the following power series, around $t=1-r-\ell$,
		\begin{align}
		&\frac{\Gamma(t+r-1)}{\Gamma(t)}= (-1)^{r-1}\frac{ (r+\ell-1)!}{\ell!}-(-1)^{r}\frac{ (r+\ell-1)!}{\ell!} (\psi(\ell+1)-\psi(r+\ell)) (t-(1-r-\ell))  \notag \\
		&-(-1)^{r}\frac{ (r+\ell-1)!}{2\ell!} \left( (\psi(\ell+1)-\psi(r+\ell))^2 - \psi'(\ell+1)+\psi'(r+\ell)  \right)(t-(1-r-\ell))^2 + O(|(t-(1-r-\ell))|^3), \label{Glim4}
		\end{align}
		\begin{align}\label{xlim5}
		\frac{1}{x^{t+r-1}} = x^{\ell}-x^{\ell} \log(x) (t-(1-r-\ell)) + \frac{1}{2} x^{\ell} \log^2(x) (t-(1-r-\ell))^2 +O(|(t-(1-r-\ell))|^3).
		\end{align}
		Use \eqref{Glim4} and \eqref{xlim5} in \eqref{M1N} to get the required result \eqref{T2C4}. 
	\end{proof}
	Next, we are interested to get Kronecker limit type formula for $\Theta(r,s,t,x)$ in the \emph{third variable} $t$, around $t= 2-r-s$. We can obtain such a result for fixed $r,s \in \mathbb{C}$ such that $r+s \in \mathbb{N}$ and $r+s \ge 2$ in the following way: 
	\begin{itemize}
		\item $r, s \in \mathbb{N}$: It reduces to case IV of Theorem \ref{2ndpolar} upon taking $s=\ell+1$.
		\item $s \in \mathbb{N}$ and $r \in \mathbb{Z} \setminus \mathbb{N}$: It reduces to case II of Theorem \ref{1stpolar} upon taking $s=\ell+1$.
		\item $r \in \mathbb{N}$ and $s \in \mathbb{Z} \setminus \mathbb{N}$: It reduces to the previous case, $s \in \mathbb{N}$ and $r \in \mathbb{Z} \setminus \mathbb{N}$, after using \eqref{Tinv}.
	\end{itemize}
	Hence, for a fixed $r,s \in  \mathbb{C} \setminus \mathbb{Z}$, we state the formula as the following theorem.  
	
	\begin{theorem}\label{3rdpolar}
		Let  $r,s \in  \mathbb{C} \setminus \mathbb{Z}$ such that $r+s \in \mathbb{N}$ and $r+s \ge 2$. Around $t=2-r-s$, we have
		\begin{align}
			\Theta(r,s,t,x) & = (-1)^{r+s} x^{s-1} (r+s-2)! \Gamma(1-r) \Gamma(1-s) \notag \\
			&\quad + \sum_{k=0}^{r+s-2} \binom{r+s-2}{k} x^{r+s-2-k}  \zeta(r-k) \zeta(k-r+2) + O\left( |t-(2-r-s)| \right). \label{T3}
		\end{align}
	\end{theorem}
	\begin{proof}
		Then, for $\textup{Re}(t)>1-r-s$, \eqref{BOr} and \eqref{BOs} gives, 
		\begin{align}
			y^{t-1} \textup{Li}_s (e^{-xy}) \bigg( \hspace{-0.2 cm} \textup{ Li}_r (e^{-y}) - \Gamma(1-r) y^{r-1} - \sum_{\substack{k=0 }}^{r+ s+\lfloor|\textup{Re}(s)|\rfloor } (-1)^k\zeta(r-k) \frac{y^k}{k!}\bigg) = O\left(y^{\epsilon-1}\right), \label{BOt2}
		\end{align} 
		for some $\e >0$. For $M=r+s+\lfloor|\textup{Re}(s)|\rfloor $, from \eqref{BO1} and \eqref{BOt2}, the integral on the left-hand side of \eqref{PE} is now defined and analytic in the extended region $\textup{Re}(t)>1-r-s$  by using \cite[Theorem 2.2]{temme}. Hence, due to the pole of $\Gamma(t)$ at $t=2-r-s$, we see
		\begin{align*}
			\lim_{t \to 2-r-s} \Bigg( \frac{1}{\Gamma(t)} \int_{0}^{\infty} y^{t-1} \textup{Li}_s (e^{-xy}) \bigg( \hspace{-0.2 cm} \textup{ Li}_r (e^{-y}) - \Gamma(1-r) y^{r-1} - \sum_{\substack{k=0 }}^{r+s+ \lfloor|\textup{Re}(s)|\rfloor } (-1)^k\zeta(r-k) \frac{y^k}{k!}\bigg) \, dy \Bigg) =0.
		\end{align*} 
		Hence, from \eqref{PE}
		\begin{align}
			\Theta(r,s,t,x) &= \frac{\Gamma(1-r)\Gamma(t+r-1)}{ x^{t+r-1}\Gamma(t)}\zeta(r+s+t-1) \notag \\
			& \quad + \sum_{\substack{k=0}}^{r+s+ \lfloor|\textup{Re}(s)|\rfloor } \frac{(-1)^k\zeta(r-k)\Gamma(t+k)}{ x^{t+k}k!\Gamma(t)} \zeta(s+t+k)  + O\left( |t-(2-r-s)| \right). \label{M1C}
		\end{align}
		Since $\Gamma(t)$ has a pole at $t=2-r-s$, only the terms with a pole in the numerator survive, as $t \to 2-r-s$. We list the poles of numerator at $t=2-r-s$: Note that, $\Gamma(t+k)$ has a simple pole only for $k \leq r+s-2$. For any $k$, $\zeta(s+t+k)$ has no pole since $r \notin \mathbb{Z}$ . $\Gamma(t+r-1)$ has no pole since $s \notin \mathbb{Z}$. $\zeta(r+s+t-1)$ has a simple pole. Hence, use the laurent series expansions of the functions on the right-hand side of \eqref{M1C} to get the required result \eqref{T3}.
	\end{proof}
	
	\section{Applications in modular relations}\label{Section3}

We start this Section by proving a recursive formula, Theorem \ref{molty}, for the generalized Mordell-Tornheim zeta function $\Theta(r,s,t,x)$. The motivation for this result originated from a recursive result for the Mordell-Tornheim zeta function $\zeta_{\textup{MT}}(r,s,t)$ obtained by Huard, Williams and Zhang \cite[Equation (1.6)]{huard}. We use the specific case of Theorem \ref{molty}, with $s=r=n$, as stated in Corollary \ref{Trrcor}, to then obtain the two-term functional equation given by Vlasenko-Zagier \eqref{vz2term} as a further implication.
		\begin{theorem}\label{molty}
		For any $r,s,t \in \mathbb{R}$ such that $r+s+t>2$, $r+t>1$ and $s+t>1$ and for any $n \in \mathbb{N}$, we have,
		\begin{align}
			\Theta(r,s,t,x) = \sum_{\ell=0}^{n} \binom{n}{\ell} x^{\ell} \Theta(r-n+\ell,s-\ell,t+n,x). \label{molltype1}
		\end{align}
	\end{theorem}
	\begin{proof}
		For $n=0$, it is vacuously true. For $n=1$, we can easily see that,
		\begin{align}
			\Theta(r,s,t,x) &= \sum_{n=1}^{\infty} \sum_{m=1}^{\infty} \frac{1}{n^{r-1} m^s (n+mx)^{t+1}} + x \sum_{n=1}^{\infty} \sum_{m=1}^{\infty} \frac{1}{n^r m^{s-1} (n+mx)^{t+1}} \notag \\
			&= \Theta(r-1,s,t+1,x)+x \Theta(r,s-1,t+1,x). \label{n=1 sum}
		\end{align}
		Hence, the statement holds for $n=1$. We now assume the identity is true for all $j<n$ and then prove that the statement holds for $j=n$. Since the statement holds for $j=n-1$, we have,
		\begin{align*}
			\Theta(r,s,t,x) = \sum_{\ell=0}^{n-1} \binom{n-1}{\ell} x^{\ell} \Theta(r-(n-1)+\ell,s-\ell,t+n-1,x).
		\end{align*}
		Use \eqref{n=1 sum} inside the summation to get,
		\begin{align*}
			&\Theta(r,s,t,x) = \sum_{\ell=0}^{n-1} \binom{n-1}{\ell} x^{\ell} \left( \Theta(r-n+\ell,s-\ell,t+n,x)+ x \Theta(r-(n-1)+\ell,s-\ell-1,t+n,x) \right) \\
			&= \sum_{\ell=0}^{n-1} \binom{n-1}{\ell} x^{\ell}  \Theta(r-n+\ell,s-\ell,t+n,x)+ x \sum_{\ell=0}^{n-1} \binom{n-1}{\ell} x^{\ell} \Theta(r-(n-1)+\ell,s-\ell-1,t+n,x)
		\end{align*}
		Replace the variable of summation $\ell$ by $\ell-1$ in the second term to get,
		\begin{align*}
			&\Theta(r,s,t,x) = \sum_{\ell=0}^{n-1} \binom{n-1}{\ell} x^{\ell}  \Theta(r-n+\ell,s-\ell,t+n,x)+ \sum_{\ell=1}^{n} \binom{n-1}{\ell-1} x^{\ell} \Theta(r-n+\ell,s-\ell,t+n,x)\\
			&= \Theta(r-\ell,s,t+n,x) + \sum_{\ell=1}^{n-1} \left( \binom{n-1}{\ell} + \binom{n-1}{\ell-1} \right) x^{\ell}  \Theta(r-n+\ell,s-\ell,t+n,x) + x^n \Theta(r,s-n,t+n,x) \\
			&= \Theta(r-\ell,s,t+n,x) + \sum_{\ell=1}^{n-1} \binom{n}{\ell}  x^{\ell}  \Theta(r-n+\ell,s-\ell,t+n,x) + x^n \Theta(r,s-n,t+n,x) \\
			&= \sum_{\ell=0}^{n} \binom{n}{\ell}  x^{\ell}  \Theta(r-n+\ell,s-\ell,t+n,x),
		\end{align*}
		where we have used the property $\binom{n-1}{\ell} + \binom{n-1}{\ell-1} = \binom{n}{\ell}$  of the Binomial coefficients in the second last step. Observe that $\Theta$ by definition is well defined at the arguments in all of the above steps. Hence, by the Principle of Mathematical Induction, \eqref{molltype1} is true for all $n \in \mathbb{N}$.
	\end{proof}
	\begin{corollary}\label{Trrcor}
		For $r,t \in \mathbb{R}$ such that $r+t>1$ and $2r+t>2$, and for any $x>0$,
		\begin{align*}
			\Theta(r,r,t,x) = \sum_{\ell=0}^{r} \binom{r}{\ell} x^{\ell} \Theta(\ell,r-\ell,t+r,x).
		\end{align*}
	\end{corollary}
	\begin{proof}
		Put $s=r$ and $n=r$ in \eqref{molltype1} to get this result.
	\end{proof}
	We now obtain the functional equations proved by Vlasenko and Zagier \cite[Equation 11]{vz} as an application of the above Corollary \ref{Trrcor}.
	\subsection{Vlasenko and Zagier's functional equation}\label{VZsection}
	\begin{corollary}
		Functional equation in \eqref{vz2term} given by Vlasenko-Zagier holds. 
	\end{corollary}
	\begin{proof}
		Let $r \in \mathbb{N}$, $r\ne1$ be arbitrary and $t>1-r$. Let $x>0$. From Corollary \ref{Trrcor}, we have,
		\begin{align*}
			\Theta(r,r,t,x) = \sum_{\ell=0}^{r} \binom{r}{\ell} x^{\ell} \Theta(\ell,r-\ell,t+r,x).
		\end{align*}
		Separate out the first and the last term from the sum to get,
		\begin{align*}
			\Theta(r,r,t,x) =\Theta(0,r,t+r,x) + \sum_{\ell=1}^{r-1} \binom{r}{\ell} x^{\ell} \Theta(\ell,r-\ell,t+r,x) + x^r \Theta(r,0,t+r,x).
		\end{align*}
		Use inversion \eqref{Tinv} for the third term on the right-hand side to get,
		\begin{align}
			\Theta(r,r,t,x) =\Theta(0,r,t+r,x) + S(r,t) + x^{-t} \Theta(0,r,t+r,\tfrac{1}{x}), \label{rlevel}
		\end{align}
		where $S(r,t)$ is defined by
		\begin{align*}
			S(r,t):= \sum_{\ell=1}^{r-1} \binom{r}{\ell} x^{\ell} \Theta(\ell,r-\ell,t+r,x).
		\end{align*}
		We now simplify $S(r,t)$ as follows:
		\begin{align*}
			S(r,t)&= \sum_{\ell=1}^{r-1} \left( \sum_{i=1}^{\ell} (-1)^{\ell-i} \binom{r}{i} -  \sum_{j=1}^{\ell-1} (-1)^{\ell-j} \binom{r}{j} \right) x^{\ell} \Theta(\ell,r-\ell,t+r,x) \\
			&= \sum_{\ell=1}^{r-1} \sum_{i=1}^{\ell} (-1)^{\ell-i} \binom{r}{i}  x^{\ell} \Theta(\ell,r-\ell,t+r,x) - \sum_{\ell=1}^{r-1} \sum_{j=1}^{\ell-1} (-1)^{\ell-j} \binom{r}{j} x^{\ell} \Theta(\ell,r-\ell,t+r,x).
		\end{align*}
		Replace $\ell$ by $\ell+1$ in the second term of the right-hand side to get,
		\begin{align*}
			S(r,t)&= \sum_{\ell=1}^{r-1} \sum_{i=1}^{\ell} (-1)^{\ell-i} \binom{r}{i}  x^{\ell} \Theta(\ell,r-\ell,t+r,x) + \sum_{\ell=0}^{r-2} \sum_{j=1}^{\ell} (-1)^{\ell-j} \binom{r}{j} x^{\ell+1} \Theta(\ell+1,r-\ell-1,t+r,x).
		\end{align*}
		Since $\ell=0$ makes the summand zero, hence the second term can be rewritten as,
		\begin{align*}
			S(r,t)&= \sum_{\ell=1}^{r-1} \sum_{i=1}^{\ell} (-1)^{\ell-i} \binom{r}{i}  x^{\ell} \Theta(\ell,r-\ell,t+r,x) + \sum_{\ell=1}^{r-2} \sum_{j=1}^{\ell} (-1)^{\ell-j} \binom{k}{j} x^{\ell+1} \Theta(\ell+1,r-\ell-1,t+r,x).
		\end{align*}
		Separate the $\ell=r-1$ term from the first term on the right-hand side to get,
		\begin{align}
			S(r,t)&= x^{r-1} \Theta(r-1,1,t+r,x) \sum_{i=1}^{r-1}(-1)^{r-i+1}  \binom{r}{i} \notag \\
			&+\sum_{\ell=1}^{r-2} \sum_{i=1}^{\ell} (-1)^{\ell-i} \binom{r}{i}  x^{\ell} \Theta(\ell,r-\ell,t+r,x) + \sum_{\ell=1}^{r-2} \sum_{j=1}^{\ell} (-1)^{\ell-j} \binom{k}{j} x^{\ell+1} \Theta(\ell+1,r-\ell-1,t+r,x)\notag\\
			&= x^{r-1} \Theta(r-1,1,t+r,x) \sum_{i=1}^{r-1}(-1)^{r-i+1}  \binom{r}{i}\notag\\
			&+\sum_{\ell=1}^{r-2} \sum_{i=1}^{\ell} (-1)^{\ell-i} \binom{r}{i}  x^{\ell} \bigg( \Theta(\ell,r-\ell,t+r,x) +  x \Theta(\ell+1,r-\ell-1,t+r,x) \bigg)\notag\\
			&= x^{r-1} \Theta(r-1,1,t+r,x) \sum_{i=1}^{r-1}(-1)^{r-i+1}  \binom{r}{i}+\sum_{\ell=1}^{r-2} \sum_{i=1}^{\ell} (-1)^{\ell-i} \binom{r}{i}  x^{\ell}  \Theta(\ell+1,r-\ell,t+r-1,x), \label{srt simp}
		\end{align}
		where we have used \eqref{Tsplit} in the last step. We note the following fact obtained by the Binomial theorem,
		\begin{align*}
			\sum_{i=1}^{r-1}(-1)^{r-i+1}\binom{r}{i}  \buildrel \rm \emph{i} \rightarrow \emph{r}-\emph{i} \over =  -\sum_{i=1}^{r-1}(-1)^{i}\binom{r}{i} =1+(-1)^r = \left\{
			\begin{array}{ll}
				0, & \text{if r is odd}, \\
				2, & \text{if r is even}. \\
			\end{array} 
			\right.
		\end{align*}
		Also, we know that 
		\begin{align*}
			\sum_{i=1}^{\ell}  (-1)^{\ell-i} \binom{r}{i}  = 	(-1)^{\ell} \sum_{i=1}^{\ell}  (-1)^{i}  \binom{r}{i} = (-1)^{\ell} \left(-1+	\sum_{i=0}^{\ell}  (-1)^{i}  \binom{r}{i} \right) = (-1)^{\ell+1} + \binom{r-1}{\ell},
		\end{align*}
		where we have used that $\sum_{i=0}^{\ell}  (-1)^{i}  \binom{r}{i} = (-1)^{\ell}\binom{r-1}{\ell}$.  Hence, using the two facts mentioned above, \eqref{srt simp} becomes,
		\begin{align}
			S(r,t)= (1+(-1)^r)x^{r-1} \Theta(r-1,1,t+r,x) +\sum_{\ell=1}^{r-2} \left( (-1)^{\ell+1} + \binom{r-1}{\ell} \right)  x^{\ell}  \Theta(\ell+1,r-\ell,t+r-1,x). \label{srt0eqn}
		\end{align}
		Take the limit $t \to 1-r$ in \eqref{srt0eqn} to get,
		\begin{align}
			\lim_{t\to 1-r} S(r,t)= (1+(-1)^r)x^{r-1} \lim_{t\to 1-r}  \Theta(r-1,1,t+r,x) +\sum_{\ell=1}^{r-2} \left( (-1)^{\ell+1} + \binom{r-1}{\ell} \right)  x^{\ell} \zeta(\ell+1) \zeta(r-\ell), \label{srt1eqn}
		\end{align}
		since $\Theta(a,b,0,x)=\zeta(a)\zeta(b)$ for any $a,b>1$. The limit on right-hand side is dealt as follows,
		\begin{align}
			\lim_{t\to 1-r}  \Theta(r-1,1,t+r,x) &= \lim_{t \to 1-r}  x^{-t-r}\Theta(1,r-1,t+r,\tfrac{1}{x}) \notag \\ &=\lim_{t \to 1-r} \left( x^{-t-r} \sum_{n=1}^{\infty} \sum_{m=1}^{\infty} \frac{1}{n m^{r-1} (n+\frac{m}{x})^{t+r}} \right) \notag \\ &= \frac{1}{x} \sum_{n=1}^{\infty} \sum_{m=1}^{\infty} \frac{1}{n m^{r-1} (n+\frac{m}{x})} \notag \\
			&= \frac{1}{x} \sum_{m=1}^{\infty} \frac{1}{m^{r-1}} \sum_{n=1}^{\infty}  \frac{1}{n (n+\frac{m}{x})} \notag\\
			&= \frac{1}{x} \sum_{m=1}^{\infty} \frac{1}{m^{r-1}} \sum_{n=1}^{\infty} \left( \frac{\frac{x}{m}}{n} - \frac{\frac{x}{m}}{n+\frac{m}{x}} \right)  \notag\\
			&=  \sum_{m=1}^{\infty} \frac{1}{m^{r}} \left( \gamma + \psi(\tfrac{m}{x}+1) \right)  \notag\\
			&= \gamma \zeta(r) + \sum_{m=1}^{\infty} \frac{\psi(\frac{m}{x}+1)}{m^{r}} \label{Tlim}
		\end{align}
		Hence, from \eqref{srt1eqn} and \eqref{Tlim}, we get,
		\begin{align}
			\lim_{t\to 1-r} S(r,t) &= (1+(-1)^r)x^{r-1} \left( \sum_{m=1}^{\infty} \frac{\psi(\frac{m}{x}+1)}{m^{r}} +\gamma \zeta(r) \right)  +\sum_{\ell=1}^{r-2} \left( (-1)^{\ell+1} + \binom{r-1}{\ell} \right)  x^{\ell} \zeta(\ell+1) \zeta(r-\ell). \label{srtLim}
		\end{align}
		We also need the following evaluation,
		\begin{align}
			\Theta(0,r,t+r,x)&= \sum_{n=1}^{\infty} \sum_{m=1}^{\infty} \frac{1}{m^r (n+mx)^{t+r}} \notag \\
			&=  \sum_{m=1}^{\infty} \frac{1}{m^r}  \sum_{n=1}^{\infty} \frac{1}{(n+mx)^{t+r}} \notag \\ 
			&= \sum_{m=1}^{\infty} \frac{\zeta(t+r,mx+1)}{m^r} \notag \\
			&= \sum_{m=1}^{\infty} \frac{\zeta(t+r,mx+1) -\frac{1}{t-1+r}}{m^r} + \frac{\zeta(r)}{t-1+r} \label{Teval}
		\end{align}
		Hence, we can also see that,
		\begin{align}
			x^{-t} \Theta(0,r,t+r,\tfrac{1}{x}) = x^{-t} \left( \sum_{m=1}^{\infty} \frac{\zeta(t+r,\tfrac{m}{x}+1) -\frac{1}{t-1+r}}{m^r} + \frac{\zeta(r)}{t-1+r} \right). \label{Teval1}
		\end{align} 
		Substitute \eqref{Teval} and \eqref{Teval1} in \eqref{rlevel} to get,
		\begin{align*}
			\Theta(r,r,t,x) - \frac{\zeta(r) (1+x^{-t})}{t-1+r} &= \sum_{m=1}^{\infty} \frac{\zeta(t+r,mx+1) -\frac{1}{t-1+r}}{m^r}  + S(r,t) \\& + x^{-t}  \sum_{m=1}^{\infty} \frac{\zeta(t+r,\tfrac{m}{x}+1) -\frac{1}{t-1+r}}{m^r} .
		\end{align*}
		Take the limit of both sides of the above equation as $t \to 1-r$, and use \eqref{srtLim} to get
		\begin{align}
			&\lim_{t \to 1-r} \left( \Theta(r,r,t,x) - \frac{\zeta(r) (1+x^{-t})}{t-1+r} \right) \notag \\ &= -\sum_{m=1}^{\infty} \frac{\psi(mx+1)}{m^r} + \lim_{t \to 1-r} (S(r,t)) - x^{r-1} \sum_{m=1}^{\infty} \frac{\psi(\frac{m}{x}+1)}{m^r}\notag \\
			&= -\sum_{m=1}^{\infty} \frac{\psi(mx+1)}{m^r} +(1+(-1)^r)x^{r-1} \left( \sum_{m=1}^{\infty} \frac{\psi(\frac{m}{x}+1)}{m^{r}} +\gamma \zeta(r) \right) \notag \\ 
			& +\sum_{\ell=1}^{r-2} \left((-1)^{\ell+1} + \binom{r-1}{\ell} \right)  x^{\ell} \zeta(\ell+1) \zeta(r-\ell) - x^{r-1} \sum_{m=1}^{\infty} \frac{\psi(\frac{m}{x}+1)}{m^r}\notag \\
			&= -\sum_{m=1}^{\infty} \frac{\psi(mx+1)}{m^r} + (-1)^r x^{r-1} \sum_{m=1}^{\infty} \frac{\psi(\frac{m}{x}+1)}{m^{r}} +(1+(-1)^r)x^{r-1} \gamma \zeta(r) \notag\\
			& +\sum_{\ell=1}^{r-2} \left((-1)^{\ell+1} + \binom{r-1}{\ell} \right)  x^{\ell} \zeta(\ell+1) \zeta(r-\ell). \label{aftLim}
		\end{align}
		Put $s=r$ and $\ell=0$ in case 3 of Theorem \ref{2ndpolar} to get,
		\begin{align}
			\Theta(r,r,t,x) &= \frac{\zeta(r)\left(1 + x^{r-1} \right)}{(t-1+r)}  +  x^{r-1} \zeta(r)(\gamma-\log(x)) \notag \\
			& + \gamma \zeta(r)  + x^{r-1} \sum_{\substack{k=1}}^{r-2} \binom{r-1}{k}  \frac{1}{ x^{k}} \zeta(r-k)\zeta(k+1)   + O(|t-1+r|).  \notag
		\end{align}
		Upon adding and subtracting suitable terms, we get,
		\begin{align}
			&\lim_{t \to 1-r} \left( \Theta(r,r,t,x) - \frac{\zeta(r) (1+x^{-t})}{t-1+r} \right) \notag \\
			&= \lim_{t \to 1-r} \left( \frac{\zeta(r)\left(x^{r-1}-x^{-t}\right)}{t-1+r} \right)  +  x^{r-1} \zeta(r)(\gamma-\log(x))  + \gamma \zeta(r)  + x^{r-1} \sum_{\substack{k=1}}^{r-2} \binom{r-1}{k}  \frac{1}{ x^{k}} \zeta(r-k)\zeta(k+1) \notag \\
			&= x^{r-1} \zeta(r) \log(x) +  x^{r-1} \zeta(r)(\gamma-\log(x)) + \gamma \zeta(r)  + x^{r-1} \sum_{\substack{k=1}}^{r-2} \binom{r-1}{k}  \frac{1}{ x^{k}} \zeta(r-k)\zeta(k+1) \notag \\
			&=  x^{r-1} \zeta(r)\gamma + \zeta(r)  \gamma  + x^{r-1} \sum_{\substack{k=1}}^{r-2} \binom{r-1}{k}  \frac{1}{ x^{k}} \zeta(r-k)\zeta(k+1) \notag \\
			&= x^{r-1} \zeta(r)\gamma + \zeta(r)  \gamma  +  \sum_{\substack{k=1}}^{r-2} \binom{r-1}{k}  x^{k} \zeta(r-k)\zeta(k+1), \label{Tconst}
		\end{align}
		where we have replaced $k$ by $r-1-k$ in the last step.
		Hence, from \eqref{aftLim} and \eqref{Tconst}, we see,
		\begin{align*}
			&-\sum_{m=1}^{\infty} \frac{\psi(mx+1)}{m^r} + (-1)^r x^{r-1} \sum_{m=1}^{\infty} \frac{\psi(\frac{m}{x}+1)}{m^{r}} +(1+(-1)^r)x^{r-1} \gamma \zeta(r) \notag\\
			& +\sum_{\ell=1}^{r-2} \left( (-1)^{\ell+1} + \binom{r-1}{\ell} \right)  x^{\ell} \zeta(\ell+1) \zeta(r-\ell) = x^{r-1} \zeta(r)\gamma + \zeta(r)  \gamma  +  \sum_{\substack{k=1}}^{r-2} \binom{r-1}{k}  x^{k} \zeta(r-k)\zeta(k+1).
		\end{align*}
		On cancellation and simplification, we get, with the convention that $\zeta(1)=\gamma$ as taken by \emph{Vlasenko-Zagier},
		\begin{align*}
			\sum_{m=1}^{\infty} \frac{\psi(mx+1)}{m^r} + (-x)^{r-1} \sum_{m=1}^{\infty} \frac{\psi(\frac{m}{x}+1)}{m^{r}} &= (-1)^{r} x^{r-1} \gamma \zeta(r) -\sum_{\ell=1}^{r-2} (-1)^{\ell}  x^{\ell} \zeta(\ell+1) \zeta(r-\ell)   -\zeta(r)  \gamma. \\
			&= -\sum_{\ell=0}^{r-1} (-x)^{\ell} \zeta(\ell+1) \zeta(r-\ell).
		\end{align*}
		Use the functional equation $\psi(y+1)=\psi(y)+ y^{-1}$ in both the summands on the left-hand side of the above equation and simplify to get \eqref{vz2term}.
	\end{proof}
	\subsection{Guinand's functional equations}
	\begin{corollary}
		Functional equations in \eqref{guinand2} and \eqref{guinand1} given by Guinand hold.
	\end{corollary}
	\begin{proof}
		We use the inversion formula \eqref{Tinv} with $r=s=0$ and $t \in \mathbb{R}$ such that $t>2$,
		\begin{align}
			\Theta(0,0,t,x) = x^{-t} \Theta(0,0,t,\tfrac{1}{x}). \label{step1}
		\end{align}
		From \cite[Equation 6.4.10, p.260]{Handbook}, for $j \in \mathbb{N}$ and $\psi^{(j)}$ being the $j^{th}$ derivative of $\psi$, we have,
		\begin{align}
			\psi^{(j)}(z) = (-1)^{j+1} j! \sum_{\ell=0}^{\infty} \frac{1}{(\ell+z)^{j+1}} = (-1)^{j+1} j! \zeta(j+1,z), \label{psi der}
		\end{align}
		where $\zeta(s,a)$ is the Hurwitz zeta function. Then, from \eqref{step1}, we have
		\begin{align}
			\sum_{n=1}^{\infty} \sum_{m=1}^{\infty} \frac{1}{(n+mx)^{t}} = x^{-t} \sum_{n=1}^{\infty} \sum_{m=1}^{\infty} \frac{1}{(n+\tfrac{m}{x})^{t}}. \label{step2}
		\end{align}
		Replace $n$ by $n-1$ in both the sums on left-hand and right-hand side of the equation and then use \eqref{psi der} to see that, for $t \in \mathbb{N}$ such that $t \geq 3$,
		\begin{align*}
			(-1)^{t}(t-1)!\sum_{m=1}^{\infty} \psi^{(t-1)}(1+mx)
			=(-1)^{t}(t-1)!x^{-t}\sum_{m=1}^{\infty} \psi^{(t-1)}\left(1+\frac{m}{x}\right).
		\end{align*}
		Cancel $(-1)^t (t-1)!$ on both sides and then multiply both sides by $x^{\frac{t}{2}}$ to get \eqref{guinand2}. We now prove the second equation of Guinand \eqref{guinand1}. From the definition of the Hurwitz zeta function, $\zeta(s,a) = \sum_{n=0}^{\infty} (n+a)^{-s}$, and \eqref{step2}, we can see,
		\begin{align}
			\sum_{m=1}^{\infty} \zeta(t,mx+1) = x^{-t} 	\sum_{m=1}^{\infty} \zeta(t,\tfrac{m}{x}+1). \label{Hzeta}
		\end{align}
		From \cite[Equation 4.3]{Analogues}, we have,
		\begin{align*}
			\zeta(z,x) = \frac{x^{1-z}}{z-1} + O(x^{-z}).
		\end{align*}
		Adding and subtracting suitable terms from the summands on both left-hand side and right-hand side,
		\eqref{Hzeta} becomes,
		\begin{align*}
			\sum_{m=1}^{\infty} \left( \zeta(t,mx+1) + \frac{(mx)^{1-t}}{t-1} - \frac{(mx)^{1-t}}{t-1}  \right) = x^{-t} \sum_{m=1}^{\infty} \left(\zeta(t,\tfrac{m}{x}+1) + \frac{\left(\frac{m}{x}\right)^{1-t}}{t-1} - \frac{\left(\frac{m}{x}\right)^{1-t}}{t-1} \right).
		\end{align*}
		Simplify to get,
		\begin{align}
			\sum_{m=1}^{\infty} \left( \zeta(t,mx+1)  - \frac{(mx)^{1-t}}{t-1} \right) + \frac{x^{1-t}\zeta(t-1)}{t-1} = x^{-t} \sum_{m=1}^{\infty} \left(\zeta(t,\tfrac{m}{x}+1) - \frac{\left(\frac{m}{x}\right)^{1-t}}{t-1} \right) + \frac{\zeta(t-1)}{x(t-1)}. \label{b4lim}
		\end{align}
		See the following limits hold, where we use \eqref{psi der} for the first one
		\begin{align}
			\lim_{t \to 2} \left(  \zeta(t,y+1)  - \frac{y^{1-t}}{t-1}  \right) = \zeta(2,y+1) - \frac{1}{y} = \psi'(1+y) - \frac{1}{y}, \label{L1}
		\end{align}
		\begin{align}
			\lim_{t \to 2} \left( \left( x^{1-t} - \frac{1}{x} \right) \frac{\zeta(t-1)}{t-1} \right) = -\frac{\log(x)}{x} = - 2 \frac{\log(x)}{2x}= -  \frac{\log(x)}{2x} + \frac{\log(\frac{1}{x})}{2x} . \label{L2}
		\end{align}
		Tend $t \to 2$ in \eqref{b4lim}, use \eqref{L1} and \eqref{L2} and then multiply both the left-hand side and right-hand side of the resultant by $x$ to get \eqref{guinand1}.
	\end{proof}
	\subsection{Ramanujan's functional equation}
	\begin{corollary}
		The first equality in the functional equation given by Ramanujan in \eqref{w1.26} holds.
	\end{corollary}
	\begin{proof}
		From \eqref{Hzeta} we have,
		\begin{align*}
			\sum_{m=1}^{\infty} \zeta(t,mx+1) = x^{-t} 	\sum_{m=1}^{\infty} \zeta(t,\tfrac{m}{x}+1).
		\end{align*}
		From \cite[Equation 4.3]{Analogues}, we have,
		\begin{align*}
			\zeta(z,x) = \frac{x^{1-z}}{z-1} + \frac{1}{2}x^{-z} + O(x^{-z-1}).
		\end{align*}
		Adding and subtracting suitable terms from the summands on both left-hand side and right-hand side, we deduce,
		\begin{align}
			&\sum_{m=1}^{\infty} \left( \zeta(t,mx+1) + \frac{(mx)^{1-t}}{t-1} + \frac{1}{2}(mx)^{-t} - \frac{(mx)^{1-t}}{t-1} - \frac{1}{2}(mx)^{-t} \right) \notag \\
			&= x^{-t} 	\sum_{m=1}^{\infty} \left(\zeta(t,\tfrac{m}{x}+1) + \frac{\left(\frac{m}{x}\right)^{1-t}}{t-1} + \frac{1}{2}\left(\frac{m}{x}\right)^{-t} - \frac{\left(\frac{m}{x}\right)^{1-t}}{t-1} - \frac{1}{2}\left(\frac{m}{x}\right)^{-t} \right). \label{almDixit}
		\end{align}
		Then,  on using $\zeta(z,a+1) = \zeta(z,a) - a^{-z}, $\eqref{almDixit} becomes,
		\begin{align*}
			&\sum_{m=1}^{\infty} \left( \zeta(t,mx) + \frac{(mx)^{1-t}}{t-1} + \frac{1}{2}(mx)^{-t} - \frac{(mx)^{1-t}}{t-1} - \frac{1}{2}(mx)^{-t} - (mx)^{-t} \right) \\
			&= x^{-t} \sum_{m=1}^{\infty} \left(\zeta(t,\tfrac{m}{x}) + \frac{\left(\frac{m}{x}\right)^{1-t}}{t-1} + \frac{1}{2}\left(\frac{m}{x}\right)^{-t} - \frac{\left(\frac{m}{x}\right)^{1-t}}{t-1} - \frac{1}{2}\left(\frac{m}{x}\right)^{-t} - \left(\frac{m}{x}\right)^{-t} \right).
		\end{align*}
		Let us denote $\varphi(t,x)$ as follows:
		\begin{align*}
			\varphi(t,x) := \zeta(t,x) - \frac{x^{1-t}}{t-1} - \frac{1}{2}x^{-t}.
		\end{align*}
		Then, we see,
		\begin{align*}
			\left(\sum_{m=1}^{\infty} \varphi(t,mx)  + \frac{x^{1-t}}{t-1} \zeta(t-1) - \frac{1}{2} x^{-t} \zeta(t) \right)= x^{-t} \left( \sum_{m=1}^{\infty} \varphi(t,\tfrac{m}{x})  + \frac{\left(\frac{1}{x}\right)^{1-t}}{t-1} \zeta(t-1) - \frac{1}{2} \left(\frac{1}{x}\right)^{-t} \zeta(t) \right).
		\end{align*}
		Take the second term on the left-hand side to the right-hand side and vice versa, to get,
		\begin{align*}
			\left(\sum_{m=1}^{\infty} \varphi(t,mx)  - \frac{\frac{1}{x}}{t-1} \zeta(t-1) - \frac{1}{2} x^{-t} \zeta(t) \right)= x^{-t} \left( \sum_{m=1}^{\infty} \varphi(t,\tfrac{m}{x})  - \frac{x}{t-1} \zeta(t-1) - \frac{1}{2} \left(\frac{1}{x}\right)^{-t} \zeta(t) \right).
		\end{align*}
		Multiply both sides by $x^{\frac{t}{2}}$ to get the first equality of \cite[Theorem 4.1]{Analogues}. As shown in \cite[Corollary 4.2]{Analogues}, letting $t \to 1$, we get the first equality of the Ramanujan's transformation formula \eqref{w1.26}.
	\end{proof}

	\section{Kronecker limit type formula for $\Theta(r,s,t,x)$ in the Second variable}\label{Section4}
	The method used in Section \ref{Sec2} cannot be used to find Kronecker limit type formula for $\Theta(r,s,t,x)$ in the \textit{second variable} $s$. Hence, we come with an alternative method to obtain the same, which involves the series evaluations of $\Theta(r,s,t,x)$ in terms of Herglotz-Zagier function \eqref{herglotzdef} and higher Herglotz functions \eqref{higherherglotz}. The evaluation is obtained by the partial fraction method, which is very effective in the theory of multiple zeta functions. For instance, Gangl, Kaneko and Zagier \cite{partialfraction} used the partial fraction, for integers $i,j \geq 2$, 
	\begin{align*}
		\frac{1}{m^i n^j} =  \sum_{a+b=i+j} \left( \binom{a-1}{i-1} \frac{1}{(m+n)^an^b}    + \binom{a-1}{j-1}  \frac{1}{(m+n)^an^b}   \right),
	\end{align*}
	to prove the Euler decomposition formula, 
	\begin{align*}
		\zeta(i) \zeta(j) = \sum_{a+b=i+j} \left( \binom{a-1}{i-1} \zeta(a,b)  +  \binom{a-1}{j-1} \zeta(a,b)  \right),
	\end{align*} 
	where $\zeta(s_1,s_2)$ is the double zeta function as defined in \eqref{doublezetadef}. Guo and Xie \cite{guoxie} used the partial fraction method to obtain natural shuffle algebra structure for the integral representations of the multiple-zeta values. We prove the partial fraction in the following Lemma \ref{parfra} before using it to obtain the series evaluation.
		\begin{lemma}\label{parfra}
		Fix $r, t \in\mathbb{N} \cup \{0\}$ such that $r+t\geq1$. Then, for $n,y>0$, the following partial fraction holds,
		\begin{align}
			\frac{1}{n^r (n+y)^t} = (-1)^r \sum_{j=0}^{t-1} \binom{j+r-1}{j} \frac{1}{y^{j+r} (n+y)^{t-j}} + \sum_{i=0}^{r-1} \binom{i+t-1}{i} \frac{(-1)^i}{n^{r-i} y^{t+i}}. \label{parts1}
		\end{align}
		with the convention that $\binom{-1}{0}=1$, $\binom{c}{-1}=0$ for any $c \in \mathbb{Z}\backslash \{-1\}$ with an exception $\binom{-1}{-1} =1$.
	\end{lemma}
	\begin{proof}
		Start with the expression on the right-hand side of \eqref{parts1} and take least common denominator to get,
		\begin{align}
			&(-1)^r \sum_{j=0}^{t-1} \binom{j+r-1}{j} \frac{y^{t-j}(n+y)^{j}}{y^{r+t} (n+y)^{t}} + \sum_{i=0}^{r-1} \binom{i+t-1}{i} \frac{(-1)^i n^i y^{r-i}}{n^{r} y^{r+t}} \notag \\
			&=(-1)^r \sum_{j=0}^{t-1} \binom{j+r-1}{j} \frac{n^{r} y^{t-j}(n+y)^{j}}{n^{r} y^{r+t} (n+y)^{t}} + \sum_{i=0}^{r-1} \binom{i+t-1}{i} \frac{(-1)^i n^i y^{r-i} (n+y)^{t}}{n^{r} y^{r+t} (n+y)^{t}} \notag \\
			&= \frac{A(n,y,r,t)}{n^r y^{r+t} (n+y)^t}, \label{calc}
		\end{align}
		where 
		\begin{align*}
			A(n,y,r,t):= (-1)^r  n^{r} \sum_{j=0}^{t-1} \binom{j+r-1}{j} y^{t-j}(n+y)^{j} + (n+y)^{t} \sum_{i=0}^{r-1} \binom{i+t-1}{i} (-1)^i n^i y^{r-i}.
		\end{align*}
		Clearly, for a fixed $r,t \in \mathbb{N}$, $A(n,y,r,t)$ is a polynomial of two variables $n,y$.
		\begin{align}
			A(n,y,r,t)&= (-1)^r  n^{r} \sum_{j=0}^{t-1} \binom{j+r-1}{j} y^{t-j} \sum_{c=0}^{j} \binom{j}{c} n^c y^{j-c}  + \sum_{d=0}^{t} \binom{t}{d} n^d y^{t-d} \sum_{i=0}^{r-1} \binom{i+t-1}{i} (-1)^i n^i y^{r-i} \notag \\
			&=  (-1)^r   \sum_{j=0}^{t-1} \sum_{c=0}^{j} \binom{j+r-1}{j}  \binom{j}{c} n^{r+c}   y^{t-c}   + \sum_{d=0}^{t} \sum_{i=0}^{r-1}  \binom{i+t-1}{i} \binom{t}{d} (-1)^i n^{d+i} y^{t-d+r-i} \notag \\
			&=  (-1)^r  \sum_{c=0}^{t-1} n^{r+c}  \sum_{j=c}^{t-1}  \binom{j+r-1}{j}  \binom{j}{c}    y^{t-c}   + \sum_{d=0}^{t} \sum_{i=0}^{r-1}  \binom{i+t-1}{i} \binom{t}{d} (-1)^i n^{d+i} y^{t-d+r-i}.   \label{last-n}
		\end{align}
		$A(n,y,r,t)$ can be seen as a polynomial in $n$ with coefficients as polynomials in $y$. Clearly, the highest degree of $n$ possible in $A(n,y,r,t)$ is $r+t-1$, where $r,t \in \mathbb{N}$ are fixed. We now show explicitly evaluate the coefficients of $n^a$ for $1\leq a\leq r+t-1$ in two cases.\\
		\textbf{Case 1:} For $1\leq a\leq r-1$.\\
		Observe that the first term on the right-hand side of \eqref{last-n} does not contribute to the coefficient of $n^a$ since least power of $n$ possible is $r$. From the second term, the only way to get $n^a$ is when $d\leq a$ and $i=a-d$. Hence we get,
		\begin{align}
			\textup{Coefficient of }n^a \textup{ in } A(n,y,r,t) & \hspace{0.23 cm}= \sum_{d=0}^{a} \binom{a-d+t-1}{a-d} \binom{t}{d} (-1)^{a-d}y^{t+r-a} \notag \\
			&  \hspace{0.23 cm}= (-1)^{a} y^{t+r-a} \sum_{d=0}^{a} \binom{a-d+t-1}{a-d} \binom{t}{d} (-1)^d  \notag  \\
			&\buildrel \rm \emph{d} \rightarrow \emph{a}-\emph{d} \over = y^{t+r-a} \sum_{d=0}^{a} \binom{d+t-1}{d} \binom{t}{a-d} (-1)^d   \notag \\
			&\hspace{0.23 cm}=   y^{t+r-a} \sum_{d=0}^{a} \frac{(d+t-1)!}{d!(t-1)!} \cdot \frac{t!}{(a-d)!(d+t-a)!} \cdot (-1)^d   \notag  \\
			&\hspace{0.23 cm}=  \sum_{d=0}^{a} \frac{(d+t-1)!}{d!(a-d)!(d+t-a)!} (-1)^d  \notag  \\
			& \hspace{0.23 cm}= \frac{t y^{t+r-a}}{a!} \sum_{d=0}^{a} (-1)^d \binom{a}{d} p(d), \label{p(d)1}
		\end{align}
		where $p(d)= \frac{(d+t-1)!}{(d+t-a)!} = (d+t-1)(d+t-2) \cdots (d+t-a+1)$ is a polynomial in $d$ of degree $a-1$. As explained in \cite[Section 3]{Spivey} and \cite[Corollary 2]{ruiz}, the sum on the right-hand side of \eqref{p(d)1} is zero, since the polynomial $p(d)$ has a degree smaller than $a$. Hence, the coefficient of $n^a$ in $A(n,y,r,t)$ is zero for $1\leq a\leq r-1$.\\
		\textbf{Case 2:} For $r\leq a\leq r+t-1$. \\
		Observe that the term $c=a-r$ in the first term of the right-hand side of \eqref{last-n} alone contributes to the coefficient of $n^a$. The condition $ a\leq r+t-1$ ensures that $a-r \leq t-1$. From the second term of the right-hand side of \eqref{last-n}, for any $i$, $d=a-i$ alone contributes to the coefficient of $n^a$. Since  $r\leq a\leq r+t-1$, we have $d=a-i\leq t$ for any $i$. Since $i\leq r-1$, we always have $d\geq0$. Hence we get,
		\begin{align}
			\textup{Coefficient of }n^a \textup{ in } A(n,y,r,t) &= y^{t+r-a}\bigg(	A_1(n,y,r,t)+A_2(n,y,r,t)\bigg). \label{AA1A2}
		\end{align}
		where,
		\begin{align}
			A_1(n,y,r,t)&:= (-1)^r  \sum_{j=a-r}^{t-1} \binom{j+r-1}{j} \binom{j}{a-r}, \label{A1def} \\  
			A_2(n,y,r,t)&:= \sum_{i=0}^{r-1} \binom{i+t-1}{i} \binom{t}{a-i} (-1)^i. \label{A2def}
		\end{align}
		We now evaluate both $	A_1(n,y,r,t)$ and $	A_2(n,y,r,t)$ explicitly. Firstly, after expanding the binomials in $A_1(n,y,r,t)$, we get,
		\begin{align}
				A_1(n,y,r,t) &= \frac{(-1)^r}{(r-1)!(a-r)!} \sum_{j=a-r}^{t-1} \frac{(j+r-1)!}{(j+r-a)!} \notag \\
				&\hspace{-0.44 cm}\buildrel \rm \emph{j} \rightarrow \emph{i}+\emph{a}-\emph{r} \over = \frac{(-1)^r}{(r-1)!(a-r)!} \sum_{i=0}^{r+t-a-1} \frac{(i+a-1)!}{i!} \notag \\
				&= \frac{(-1)^r(a-1)!}{(r-1)!(a-r)!} \sum_{i=0}^{r+t-a-1} \binom{i+a-1}{i} \notag \\
				&= \frac{(-1)^r(a-1)!}{(r-1)!(a-r)!} \binom{r+t-1}{r+t-a-1} \notag  \\
				&=\frac{(-1)^r}{a} \times \frac{(r+t-1)!}{(r-1)!(t-1)!} \times \frac{(t-1)!}{(a-r)!(r+t-a-1)!} \notag \\
				&= \frac{(-1)^r (r+t-a)}{a} \binom{r+t-1}{t} \binom{t}{a-r}. \label{A1Fin}
		\end{align}
		Next, we evaluate $A_2(n,y,r,t)$. Since $\binom{t}{a-i}=0$ for $i<a-t$, it is enough to consider the sum between $a-t\leq i \leq r-1$, and on expanding the binomials, we get,
		\begin{align}
			A_2(n,y,r,t) &=t \sum_{i=a-t}^{r-1} \frac{(-1)^i}{i! (a-i)!} \times \frac{(i+t-1)!}{(i+t-a)!} \notag \\
			&\hspace{-0.44 cm}\buildrel \rm \emph{i} \rightarrow \emph{j}+\emph{a}-\emph{t} \over = t \sum_{j=0}^{r+t-a-1} \frac{(-1)^{j+a-t}}{(j+a-t)!(t-j)!} \times \frac{(j+a-1)!}{j!} \notag \\
			&= \frac{(-1)^{a+t}}{(t-1)!} \sum_{j=0}^{r+t-a-1} (-1)^j \binom{t}{j} \frac{(j+a-1)!}{(j+a-t)!} \notag \\
			&=  \frac{(-1)^{a+t}}{(t-1)!}  \times \frac{(-1)^{a+r+t}(a-r-t)}{at (r-1)!} \binom{t}{r+t-a} (r+t-1)! \notag \\
			&= - \frac{(-1)^r(r+t-a)}{a} \binom{r+t-1}{t} \binom{t}{a-r}, \label{A2Fin}
		\end{align}
		where we have used the fact $\binom{c}{d} = \binom{c}{c-d}$ in the last step and the following identity obtained from \textit{Mathematica 11} in the second last step,
		\begin{align*}
			\sum_{j=0}^{r+t-a-1} (-1)^j \binom{t}{j} \frac{(j+a-1)!}{(j+a-t)!}=\frac{(-1)^{a+r+t}(a-r-t)}{at (r-1)!} \binom{t}{r+t-a} (r+t-1)!.
		\end{align*}
		On substituting \eqref{A1Fin} and \eqref{A2Fin} in \eqref{AA1A2}, we can see that the coefficient of $n^a$ is zero for $r\leq a \leq r+t-1$.	Thus, both the cases combined, we have shown that the coefficient of $n^a$ is zero for all $1\leq a\leq r+t-1$, and thus, $A(n,y,r,t)$ is independent of $n$, and hence, 
		\begin{align*}
			A(n,y,r,t) = A(0,y,r,t)=y^{r+t}.
		\end{align*}
		Put this in \eqref{calc} to get the required result.
	\end{proof}	
	We now use Lemma \ref{parfra} to prove the follwing theorem for the series evaluations of $\Theta(r,s,t,x)$.
		\begin{theorem}\label{THM-EVAL}
		For $r,t \in \mathbb{N} \cup \{0\}$ such that $r+t\geq1$, and $s \in \mathbb{C}$ such that $\textup{Re}(s)>2-r-t$, we have, 
		\begin{align}
			\Theta(r,s,t,x)  &=  \sum_{i=0}^{r-2} \frac{(-1)^i}{x^{t+i}}  \binom{i+t-1}{i}  \zeta(s+t+i) \zeta(r-i) - \frac{(-1)^r}{x^{r+t-1}} \binom{r+t-2}{t-1}  \sum_{m=1}^{\infty} \frac{\gamma + \psi(mx+1)}{m^{r+s+t-1}} \notag \\
			&\quad + (-1)^r  \sum_{j=0}^{t-2} \binom{j+r-1}{j}  \frac{(-1)^{t-j}}{(t-j-1)!} \frac{1}{x^{j+r}} \sum_{m=1}^{\infty} \frac{\psi^{(t-j-1)}(mx+1)}{m^{r+s+j}}, \label{T-H-EV}
		\end{align}
		with the convention that $\binom{-1}{0}=1$, $\binom{c}{-1}=0$ for any $c \in \mathbb{Z}\backslash \{-1\}$ with an exception $\binom{-1}{-1} =1$.
	\end{theorem}
	\begin{proof}
		From the definition \eqref{Tdef}, we can see 
		\begin{align}
			\Theta(r,s,t,x)&= \sum_{n=1}^{\infty} \sum_{m=1}^{\infty} \frac{1}{n^r m^s (n+mx)^t} = \sum_{m=1}^{\infty} \frac{1}{m^s} \sum_{n=1}^{\infty} \frac{1}{n^r (n+mx)^t}. \label{EvStart}
		\end{align}
		Use Lemma \ref{parfra} with $y=mx$ in \eqref{EvStart} to get,
		\begin{align*}
			\Theta(r,s,t,x) &= \sum_{m=1}^{\infty} \frac{1}{m^s} \sum_{n=1}^{\infty} \left( (-1)^r \sum_{j=0}^{t-1} \binom{j+r-1}{j} \frac{1}{(mx)^{j+r} (n+mx)^{t-j}} + \sum_{i=0}^{r-1} \binom{i+t-1}{i} \frac{(-1)^i}{n^{r-i} (mx)^{t+i}} \right) \\
			&= \sum_{m=1}^{\infty} \frac{1}{m^s} \sum_{n=1}^{\infty} \left( (-1)^r \sum_{j=0}^{t-2} \binom{j+r-1}{j} \frac{1}{(mx)^{j+r} (n+mx)^{t-j}} + \sum_{i=0}^{r-2} \binom{i+t-1}{i} \frac{(-1)^i}{n^{r-i} (mx)^{t+i}} \right) \\
			&\quad+  \sum_{m=1}^{\infty} \frac{1}{m^s} \sum_{n=1}^{\infty} \left( (-1)^r \binom{r+t-2}{t-1} \frac{1}{(mx)^{r+t-1}(n+mx)} + (-1)^{r-1} \binom{r+t-2}{r-1} \frac{1}{n (mx)^{r+t-1}}  \right) \\
			&= \sum_{m=1}^{\infty} \frac{1}{m^s} \sum_{n=1}^{\infty} \left( (-1)^r \sum_{j=0}^{t-2} \binom{j+r-1}{j} \frac{1}{(mx)^{j+r} (n+mx)^{t-j}} + \sum_{i=0}^{r-2} \binom{i+t-1}{i} \frac{(-1)^i}{n^{r-i} (mx)^{t+i}} \right) \\
			&\quad - (-1)^r \binom{r+t-2}{t-1} \frac{1}{x^{r+t-1}} \sum_{m=1}^{\infty} \frac{1}{m^{r+s+t-1}} \sum_{n=1}^{\infty} \left( \frac{1}{n} - \frac{1}{n+mx}  \right),
		\end{align*}
		where we have used $\binom{r+t-2}{r-1} = \binom{r+t-2}{t-1}$. From the series definition of the digamma function $\psi(x)$, we have,
		\begin{align*}
			\Theta(r,s,t,x) &= \sum_{m=1}^{\infty} \frac{1}{m^s} \sum_{n=1}^{\infty} \left( (-1)^r \sum_{j=0}^{t-2} \binom{j+r-1}{j} \frac{1}{(mx)^{j+r} (n+mx)^{t-j}} + \sum_{i=0}^{r-2} \binom{i+t-1}{i} \frac{(-1)^i}{n^{r-i} (mx)^{t+i}} \right) \\
			&\quad-(-1)^r \binom{r+t-2}{t-1} \frac{1}{x^{r+t-1}} \sum_{m=1}^{\infty} \frac{\gamma + \psi(mx+1)}{m^{r+s+t-1}}.
		\end{align*} 
		From \eqref{psi der}, we have, for $j \in \mathbb{N}, j \geq 2$
		\begin{align}
			\frac{(-1)^j}{(j-1)!}\psi^{(j-1)}(z+1) = \sum_{\ell=1}^{\infty} \frac{1}{(\ell+z)^{j}}. \label{psidash}
		\end{align}
		Hence, we can simplify $\Theta(r,s,t,x)$ as
		\begin{align*}
			\Theta(r,s,t,x) &= (-1)^r  \sum_{j=0}^{t-2} \binom{j+r-1}{j} \frac{1}{x^{j+r}} \sum_{m=1}^{\infty} \frac{1}{m^{r+s+j}} \sum_{n=1}^{\infty} \frac{1}{(n+mx)^{t-j}} \\ &\quad +  \sum_{i=0}^{r-2} \frac{(-1)^i}{x^{t+i}} \binom{i+t-1}{i}  \sum_{m=1}^{\infty} \frac{1}{m^{s+t+i}} \sum_{n=1}^{\infty}  \frac{1}{n^{r-i}} -  \frac{(-1)^r}{x^{r+t-1}} \binom{r+t-2}{t-1} \sum_{m=1}^{\infty} \frac{\gamma + \psi(mx+1)}{m^{r+s+t-1}}.
		\end{align*}
		Use \eqref{psidash} to get,
		\begin{align*}
			\Theta(r,s,t,x) &= (-1)^r  \sum_{j=0}^{t-2} \binom{j+r-1}{j}  \frac{(-1)^{t-j}}{(t-j-1)!} \frac{1}{x^{j+r}} \sum_{m=1}^{\infty} \frac{\psi^{(t-j-1)}(mx+1)}{m^{r+s+j}}   \\
			&\quad +  \sum_{i=0}^{r-2} \frac{(-1)^i}{x^{t+i}}  \binom{i+t-1}{i}  \zeta(s+t+i) \zeta(r-i) - \frac{(-1)^r}{x^{r+t-1}} \binom{r+t-2}{t-1}  \sum_{m=1}^{\infty} \frac{\gamma + \psi(mx+1)}{m^{r+s+t-1}}.
		\end{align*}
		This is the required result \eqref{T-H-EV}.
	\end{proof}
	\subsection{Mixed functional equations} The two-term functional equation given by Guinand \eqref{guinand2} has higher derivatives of $\psi(x)$ in the summand while the one given by Vlasenko-Zagier \eqref{vz2term} has the higher power of $m$ in the denominator. We call the functional equation obtained in \eqref{mixed}, the mixed functional equations since they contain both, a series containing higher derivatives of $\psi(x)$ and containing higher powers of $m$ in the denominator each.
	\begin{remark}\label{REMARK1}
		The identity \eqref{newIDex} holds true.
	\end{remark}
	\begin{proof}
		Use \eqref{Tsplit} twice to get,
		\begin{align}
			\Theta(2,2,t,x) = \Theta(0,2,t+2,x) + 2x \Theta(1,1,t+2,x) + x^2 \Theta(2,0,t+2,x).  \label{remark1}
		\end{align}
		Use \eqref{Tinv} for the third term on the right-hand side of \eqref{remark1} and tend $t\to0$ to get,
		\begin{align*}
			\zeta^2(2) = \Theta(0,2,2,x) + 2x \Theta(1,1,2,x) + \Theta(0,2,2,\tfrac{1}{x}).
		\end{align*}
		Then, use Theorem \ref{THM-EVAL} thrice and simplify to get \eqref{newIDex}.
	\end{proof}
	
	 Such identitites are new in literature. We now give a family of mixed functional equations in the following theorem as a direct corollary of Theorem \ref{THM-EVAL}.
	\begin{theorem}\label{MIXEDthm}
		Let us define $\mathscr{F}(x)$ as follows,
		\begin{align*}
			\mathscr{F}(x):=& (-1)^r  \sum_{j=0}^{t-2} \binom{j+r-1}{j}  \frac{(-1)^{t-j}}{(t-j-1)!} \frac{1}{x^{j+r}} \sum_{m=1}^{\infty} \frac{\psi^{(t-j-1)}(mx+1)}{m^{2r+j}}  \notag \\ 
			&+ \sum_{i=0}^{r-2} \frac{(-1)^i}{x^{t+i}}  \binom{i+t-1}{i}  \zeta(r+t+i) \zeta(r-i) - \frac{(-1)^r}{x^{r+t-1}} \binom{r+t-2}{t-1}  \sum_{m=1}^{\infty} \frac{\gamma + \psi(mx+1)}{m^{2r+t-1}}.
		\end{align*}
		Then, we have,
		\begin{align}\label{mixed}
			\mathscr{F}(x) = x^{-t} \mathscr{F}\left(\frac{1}{x}\right)\hspace{-0.1 cm}.
		\end{align}
	\end{theorem}
	\begin{proof}
		Put $s=r$ in \eqref{T-H-EV} and substitute it in $\Theta(r,r,t,x)=x^{-t} \Theta(r,r,t,\tfrac{1}{x})$.
	\end{proof}
	\subsection{Kronecker limit type formula}
	\begin{theorem}\label{KLF in s}
		\textbf{\textup{Case 1:}} \textup{Fix an  $r \in \mathbb{N}, r \geq 2$.}\\
		Then, for any $t \in \mathbb{N} \cup \{0 \}$, $s+t=1$ is a singular point of $\Theta(r,s,t,x)$, with the following polar singularity structure in $s$, around $s=1-t$,
		\begin{align} 
			&\Theta(r,s,t,x) = \frac{x^{-t}\zeta(r)}{(s-(1-t
				))} + \Bigg(  \gamma \zeta(r) x^{-t} + (-1)^r  \sum_{j=0}^{t-2} \binom{j+r-1}{j}  \frac{(-1)^{t-j}}{(t-j-1)!} \frac{1}{x^{j+r}} \sum_{m=1}^{\infty} \frac{\psi^{(t-j-1)}(mx+1)}{m^{r+j-t+1}}   \notag \\ & +  \sum_{i=1}^{r-2} \frac{(-1)^i}{x^{t+i}}  \binom{i+t-1}{i}  \zeta(i+1) \zeta(r-i) - \frac{(-1)^r}{x^{r+t-1}} \binom{r+t-2}{t-1}  \sum_{m=1}^{\infty} \frac{\gamma + \psi(mx+1)}{m^{r}} \Bigg) + O_x(|s-(1-t)|). \label{klfs1}
		\end{align}
		\textbf{\textup{Case 2:}} \textup{$r=1$.}\\
		Then  for any $t \in \mathbb{N}$, $s+t=1$ is a singular point of $\Theta(1,s,t,x)$, with the following polar singularity structure in $s$, around $s=1-t$,
		\begin{align}
			\Theta(1,s,t,x) &= \frac{x^{-t}}{(s-(1-t))^2} + \frac{x^{-t}}{(s-(1-t))} (\gamma+\log(x)+H_{t-1}) +\bigg( \frac{1}{x^{t}}(\gamma (\gamma+\log(x)+ H_{t-1}) + \gamma_1) \notag  \\
			& \left. -\sum_{j=0}^{t-2} \frac{(-1)^{t-j}x^{-j-1}}{(t-j-1)!}  \sum_{m=1}^{\infty} \frac{\psi^{(t-j-1)}(mx) + (-1)^{t-j}(t-j-2)!(mx)^{-(t-j-1)}}{m^{j-t+2}} \right. \notag \\
			& + \frac{t\zeta(2)}{x^{t+1}} +\frac{1}{x^t} \sum_{m=1}^{\infty} \frac{\psi(mx)-\log(mx)}{m} \bigg) + O_x(|s-(1-t)|).  \label{klfs2}
		\end{align}
	\end{theorem}
	\begin{proof}
		Let us start with the case 1. From \eqref{T-H-EV}, we have,
		\begin{align*}
			\Theta(r,s,t,x)- \frac{1}{x^t} \zeta(s+t)\zeta(r) &=  (-1)^r  \sum_{j=0}^{t-2} \binom{j+r-1}{j}  \frac{(-1)^{t-j}}{(t-j-1)!} \frac{1}{x^{j+r}} \sum_{m=1}^{\infty} \frac{\psi^{(t-j-1)}(mx+1)}{m^{r+s+j}}   \notag \\ & \hspace{-0.3 cm} +  \sum_{i=1}^{r-2} \frac{(-1)^i}{x^{t+i}}  \binom{i+t-1}{i}  \zeta(s+t+i) \zeta(r-i) - \frac{(-1)^r}{x^{r+t-1}} \binom{r+t-2}{t-1}  \sum_{m=1}^{\infty} \frac{\gamma + \psi(mx+1)}{m^{r+s+t-1}}  . 
		\end{align*}
		Since we know, around $s=1-t$,
		\begin{align*}
			\zeta(s+t) = \frac{1}{s-(1-t)} +\gamma +O(|s-(1-t)|).
		\end{align*}
		Hence 
		\begin{align*}
			\lim_{s \to 1-t}	&\left( \Theta(r,s,t,x)- \frac{1}{x^t} \zeta(s+t)\zeta(r) \right) =  (-1)^r  \sum_{j=0}^{t-2} \binom{j+r-1}{j}  \frac{(-1)^{t-j}}{(t-j-1)!} \frac{1}{x^{j+r}} \sum_{m=1}^{\infty} \frac{\psi^{(t-j-1)}(mx+1)}{m^{r+j-t+1}}   \notag \\ & \hspace{2.5 cm} +  \sum_{i=1}^{r-2} \frac{(-1)^i}{x^{t+i}}  \binom{i+t-1}{i}  \zeta(i+1) \zeta(r-i) - \frac{(-1)^r}{x^{r+t-1}} \binom{r+t-2}{t-1}  \sum_{m=1}^{\infty} \frac{\gamma + \psi(mx+1)}{m^{r}}.
		\end{align*}
		Simplify to get, around $s=1-t$,
		\begin{align*} 
			&\Theta(r,s,t,x) = \frac{x^{-t}\zeta(r)}{(s-1+t)} + \Bigg(  \gamma \zeta(r) x^{-t} + (-1)^r  \sum_{j=0}^{t-2} \binom{j+r-1}{j}  \frac{(-1)^{t-j}}{(t-j-1)!} \frac{1}{x^{j+r}} \sum_{m=1}^{\infty} \frac{\psi^{(t-j-1)}(mx+1)}{m^{r+j-t+1}}   \notag \\ & +  \sum_{i=1}^{r-2} \frac{(-1)^i}{x^{t+i}}  \binom{i+t-1}{i}  \zeta(i+1) \zeta(r-i) - \frac{(-1)^r}{x^{r+t-1}} \binom{r+t-2}{t-1}  \sum_{m=1}^{\infty} \frac{\gamma + \psi(mx+1)}{m^{r}} \Bigg) + O(|s-1+t|).
		\end{align*}
		For $t-j-1\geq1$, $\psi^{(t-j-1)}(y)=O(y^{-(t-j-1)})$ as $y \to \infty$. Hence, we can see that, as $x \to \infty$, $$\frac{\psi^{(t-j-1)}(mx+1)}{m^{r+j-t+1}} = O_x\left(\frac{1}{m^r}\right).$$
		Since $r \geq 2$, the series involving the derivatives of $\psi(x)$ is absolutely convergent. This is the required result \eqref{klfs1} for the first case.
		
		For the case 2, put $r=1$ in Theorem \ref{THM-EVAL} to get,
		\begin{align}
			\Theta(1,s,t,x) &= -  \sum_{j=0}^{t-2} \frac{(-1)^{t-j}}{(t-j-1)!} \frac{1}{x^{j+1}} \sum_{m=1}^{\infty} \frac{\psi^{(t-j-1)}(mx+1)}{m^{s+j+1}}  + \frac{1}{x^{t}}  \sum_{m=1}^{\infty} \frac{\gamma + \psi(mx+1)}{m^{s+t}}. \label{r=1 eval}
		\end{align}
		We simplify the terms on the right-hand side of \eqref{r=1 eval} as follows:
		\begin{align}
			\sum_{m=1}^{\infty} \frac{\gamma + \psi(mx+1)}{m^{s+t}} &= \gamma \zeta(s+t) + \sum_{m=1}^{\infty} \frac{\psi(mx+1)}{m^{s+t}} \notag \\&= \gamma \zeta(s+t) + \frac{1}{x}\zeta(s+t+1) + \sum_{m=1}^{\infty} \frac{\psi(mx)}{m^{s+t}}  \notag \\
			&= \gamma \zeta(s+t) + \frac{1}{x}\zeta(s+t+1) + \sum_{m=1}^{\infty} \frac{\psi(mx)-\log(mx)}{m^{s+t}} + \sum_{m=1}^{\infty} \frac{\log(mx)}{m^{s+t}} \notag \\
			&= (\gamma+\log(x)) \zeta(s+t) + \frac{1}{x}\zeta(s+t+1) -\zeta'(s+t) + \sum_{m=1}^{\infty} \frac{\psi(mx)-\log(mx)}{m^{s+t}}. \label{simp1}
		\end{align} 
		Simplification of the other term is as follows: Since $t-j-1\geq1$, we have $\psi^{(t-j-1)}(y+1)=\psi^{(t-j-1)}(y)-(-1)^{t-j}(t-j-1)!y^{j-t}$. Using it we can see,
		\begin{align} 
			\sum_{m=1}^{\infty} \frac{\psi^{(t-j-1)}(mx+1)}{m^{s+j+1}}	&=- \frac{(-1)^{t-j}(t-j-1)!}{x^{t-j}} \zeta(s+t+1)+ \sum_{m=1}^{\infty} \frac{\psi^{(t-j-1)}(mx)}{m^{s+j+1}} \notag \\
			&= - \frac{(-1)^{t-j}(t-j-1)!}{x^{t-j}} \zeta(s+t+1)- \frac{(-1)^{t-j}(t-j-2)!}{x^{t-j-1}} \zeta(s+t) \notag \\
			&\quad + \sum_{m=1}^{\infty} \frac{\psi^{(t-j-1)}(mx) + (-1)^{t-j}(t-j-2)!(mx)^{-(t-j-1)}}{m^{s+j+1}}. \label{simp2}
		\end{align}
		Substitute \eqref{simp1} and \eqref{simp2} in \eqref{r=1 eval} to get,
		\begin{align*}
			&\Theta(1,s,t,x) - \frac{1}{x^t} \Bigg(\gamma+\log(x)+  \sum_{j=0}^{t-2} \frac{1}{(t-j-1)} \Bigg) \zeta(s+t) + \frac{1}{x^t} \zeta'(s+t)  \notag \\
			&= \frac{t}{x^{t+1}}\zeta(s+t+1) +\frac{1}{x^t} \sum_{m=1}^{\infty} \frac{\psi(mx)-\log(mx)}{m^{s+t}} \notag \\
			&\quad -  \sum_{j=0}^{t-2} \frac{(-1)^{t-j}}{(t-j-1)!} \frac{1}{x^{j+1}} \sum_{m=1}^{\infty} \frac{\psi^{(t-j-1)}(mx) + (-1)^{t-j}(t-j-2)!(mx)^{-(t-j-1)}}{m^{s+j+1}}.
		\end{align*}
		Taking limit as $s \to 1-t$, we get,
		\begin{align*}
			&\lim_{s \to 1-t} \left( \Theta(1,s,t,x) - \frac{1}{x^t} \left(\gamma+\log(x)+ H_{t-1} \right) \zeta(s+t) + \frac{1}{x^t} \zeta'(s+t) \right) \\
			&= \frac{t}{x^{t+1}}\zeta(2) +\frac{1}{x^t} \sum_{m=1}^{\infty} \frac{\psi(mx)-\log(mx)}{m} \notag \\
			&\quad - \sum_{j=0}^{t-2} \frac{(-1)^{t-j}}{(t-j-1)!} \frac{1}{x^{j+1}} \sum_{m=1}^{\infty} \frac{\psi^{(t-j-1)}(mx) + (-1)^{t-j}(t-j-2)!(mx)^{-(t-j-1)}}{m^{j-t+2}}.
		\end{align*}
		where $H_n$ is the $n^{th}$ Harmonic number, with the convention $H_0=0$. Also, see that, as $y \to \infty$,
		\begin{align*}
			\psi^{(t-j-1)}(y) + (-1)^{t-j}(t-j-2)! \frac{1}{y^{t-j-1}} = O(y^{-(t-j)}),
		\end{align*}
		which ensures the convergence of the series involving the derivatives of $\psi(x)$. Use the Laurent expansions of $\zeta(s+t)$ and $\zeta'(s+t)$ around $s=1-t$ to get the required result \eqref{klfs2}.
	\end{proof}
	\begin{corollary}\label{s cor}
		\begin{align*}
			\Theta(1,s,1,x) = \frac{1}{xs^2} +\frac{\gamma+\log(x)}{xs} +\frac{1}{x}(\gamma(\gamma+\log(x))+\gamma_1) +\frac{\zeta(2)}{x^{2}}  +\frac{1}{x} \sum_{m=1}^{\infty} \frac{\psi(mx)-\log(mx)}{m}.
		\end{align*}
	\end{corollary}
	\begin{proof}
		Put $r=t=1$ in Theorem \ref{KLF in s} and tend $s \to 0$ to get,
		\begin{align*}
			\lim_{s \to 0} \left( \Theta(1,s,1,x) - \frac{1}{x} \left(\gamma+\log(x) \right) \zeta(s+1) + \frac{1}{x} \zeta'(s+1) \right) = \frac{\zeta(2)}{x^{2}}  +\frac{1}{x} \sum_{m=1}^{\infty} \frac{\psi(mx)-\log(mx)}{m}.
		\end{align*}
		Use the Laurent series expansions of $\zeta(s+1)$ and $\zeta'(s+1)$ around $s=0$ to get the result. 
	\end{proof}
	\begin{remark}
		From \cite[Equation 7.12]{zagier},
		\begin{align*}
			\sum_{m=1}^{\infty} \frac{ \psi(m) -\log(m)}{m} = -\frac{\gamma^2}{2} -\frac{\pi^2}{12}  -\gamma_1,
		\end{align*}
	    and putting $x=1$ in Corollary \ref{s cor},  we can see that around $s=0$,
		\begin{align}
			\Theta(1,s,1,1) = \frac{1}{s^2} + \frac{\gamma}{s} + \frac{6\gamma^2+\pi^2}{12}    +O(|s|). \label{REMARK2}
		\end{align}
	\end{remark}

	\section{Modular relations: The Mordell-Tornheim zeta perspective}\label{SecTable}
	
	We list few of the corollaries in the Table \ref{table} below. We start with the necessary arguments for $\Theta$ in the first column. After using \eqref{Tsplit} the number of times as in column 3, we do the operation  on $t$ as in column 2, to obtain the result mentioned in column 4.
	
		For example, $\Theta(1,1,t,x)$ can give mixed functional equations, an instance of Vlasenko-Zagier's functional equation \eqref{vz2term}, based on the steps mentioned in columns 2 and 3. Observe that $\Theta(0,0,t,x)$ is the centroid between Ramanujan's functional equation \eqref{w1.26} and Guinand's functional equations \eqref{guinand2} and \eqref{guinand1}, differing by the limiting value of $t$ alone, bringing the three results under a single umbrella. Look at \cite[Corollary 3.5]{dss2} to see how Zagier's result \eqref{fe2} can be obtained.
	\begin{figure}[h!]
	\begin{center}
		\begin{tabular}{ | c| c| c| c| } 
			\hline
			$\Theta$& Operation on $t$& No. of splits&Result obtained  \\
			\hline
			\hline
			\multirow{3}{5em}{$\Theta(0,0,t,x)$} & $t \in \mathbb{N}$, $t\geq3$ &0& Guinand \eqref{guinand2} \\ 
			\cline{2-4}
			& $t \to 2$ &0& Guinand \eqref{guinand1} \\ 
			\cline{2-4}
			& $t \to 1$ &0& Ramanujan \eqref{w1.26} \\ 
			\hline
			\multirow{3}{5em}{$\Theta(1 ,1,t,x)$} & $t \in \mathbb{N}$, $t\geq2$&0& Mixed functional equations, Theorem \ref{MIXEDthm}\\
			\cline{2-4}
			& $t \to 1$ &0& Vlasenko-Zagier \eqref{vz2term} with $r=2$ \\ \cline{2-4}
			&$t \to 0$&1& Zagier \eqref{fe2} \\ 
			\hline
			\multirow{5}{5em}{$\Theta(2,2,t,x)$} & $t \in \mathbb{N}$, $t \geq 2$&0& Mixed functional equations, Theorem \ref{MIXEDthm}\\ \cline{2-4}
			&$t \to 1$&0& Vlasenko-Zagier \eqref{vz2term} with $r=4$ \\ \cline{2-4}
			&\multirow{2}{2.5em}{$t \to 0$}&1& Vlasenko-Zagier \eqref{vz2term} with $r=3$ \\ \cline{3-4}
			& &2& Mixed functional equations, Remark \ref{REMARK1} \\ \cline{2-4}
			&$t \to -1$&2& Vlasenko-Zagier \eqref{vz2term} with $r=2$\\
			\hline
		\end{tabular}
		\caption{The table shows the connections between the Mordell-Tornheim zeta function and the various modular relations in literature.}
		\label{table}
	\end{center}
	\end{figure}
	\section{Concluding Remarks}
\noindent We conclude the paper with the following remarks:
	\begin{enumerate}
\item In this paper, we have obtained the Kronecker limit formula for $\Theta(r,s,t,x)$ in the variable $t$ in Theorem \ref{1stpolar} and Theorem \ref{2ndpolar} for $\ell\geq 1-r$ and Theorem \ref{3rdpolar} for $r+s \in \mathbb{N},\,  r+s \ge 2$. We can extend Theorem \ref{1stpolar} and Theorem \ref{2ndpolar} for all $\ell \in \mathbb{N} \cup \{0\}$ and Theorem \ref{3rdpolar} for all $r,s \in \mathbb{C}$. However, the closed form expression for the constant term is a challenging task, which is in terms of Arakawa-Kaneko constants and the integral mentioned in \eqref{AK} and \eqref{Lk*} respectively.
\item We have the description of polar nature of $\Theta(r,s,t,x)$ as a function of the \emph{third variable} $t$ as well as the \emph{second variable} $s$ (hence the \emph{first variable} $r$) as a function of one variable. However, it will be very interesting to see the nature of series expansion as a function of three variables $r,s$ and $t$ simultaneously, in the sense of several complex variables.
\item As discussed in the Section \ref{SecTable}, the Mordell-Tornheim zeta function gives a new perspective of the two-term functional equations in the literature. This interplay between $\Theta(r,s,t,x)$ and the modular relations thus widens the scope for the study on $\Theta(r,s,t,x)$ and its generalizations. For example, having higher order terms of $\Theta(r,s,t,x)$ in Theorem \ref{1stpolar}, will give us two-term explicit functional equations Ishibashi and Higher Ishibashi functions, showed in \cite[Theorem 3.4]{dss2}.
\item  We have obtained two term functional equations as an corollary of Theorem \eqref{molty}. Three-term analogue of \eqref{molty}  will be an extremely desirable result. That is because, such an identity will enable us to derive three-term functional equations such as \eqref{fe1} and \eqref{vz3term}. 
\item The result \eqref{REMARK2} could have interesting applications. 
\end{enumerate}
	\section{Acknowledgment}
	The authors would like to extend gratitude to Prof. Atul Dixit for his insights and continuous support. The first author was partially supported by the grant MIS/IITGN/R\&D/MATH/AD/2324/058 of IIT Gandhinagar.

\end{document}